\documentclass[12pt]{amsart}

\usepackage{amssymb,amsmath,amsthm,amsfonts,amscd}
\usepackage{hyperref}
\usepackage{xcolor}
\usepackage{graphicx}
\usepackage[outdir=./]{epstopdf}
\usepackage{float}

\usepackage{xcolor}

\newcommand{\doi}[1]{\href{https://dx.doi.org/#1}{{\small DOI:#1}}}

\newcommand{\googlebooks}[1]{(preview at \href{https://books.google.com/books?id=#1}{google books})}

\newcommand{\numdam}[1]{}

\def\altdb{\vadjust{\vbox to 0pt{\vss\hbox{\kern \hsize
\quad{\dbend}}\kern\baselineskip\kern-10pt}}}

\theoremstyle{plain}
\newtheorem{prop}{Proposition}[section]
\newtheorem{conj}[prop]{Conjecture}
\newtheorem{thm}[prop]{Theorem}
\newtheorem{lem}[prop]{Lemma}

\numberwithin{equation}{section}
\newtheorem{rem}[prop]{Remark}
\newtheorem{defn}[prop]{Definition}

\newcommand{\id}{\operatorname{id}}
\renewcommand{\Vec}{\operatorname{Vec}}

\newcommand{\cC}{{\mathcal{C}}}
\newcommand{\Inv}{{\operatorname{Inv}}}
\newcommand{\Id}{{\operatorname{Id}}}
\newcommand{\eqbr}{{\underline{\operatorname{EqBr}}}}
\newcommand{\Pic}{{\operatorname{Pic}}}
\newcommand{\op}{\operatorname{op}}
\newcommand{\mop}{\operatorname{mop}}
\allowdisplaybreaks

\title{Fusion rules for $\mathbb{Z}/2\mathbb{Z}$ permutation gauging}
\author[Cain Edie-Michell]{Cain Edie-Michell}
\email{cain.edie-michell@vanderbilt.edu}
\address{Department of Mathematics, Vanderbilt University, Nashville, Tennessee 37212, U.S.A}
\author[Corey Jones]{Corey Jones}
\email{cormjones88@gmail.com}
\address{Department of Mathematics, The Ohio State University, Columbus, Ohio 43210, U.S.A}
\author{Julia Yael Plavnik}
\email[Julia Yael Plavnik]{jplavnik@iu.edu}
\address{Department of Mathematics, Indiana University, Bloomington, Indiana 47405, U.S.A}

\date{\today}
\begin{document}
\begin{abstract} In this note, we examine the gauging of the $\mathbb{Z}/2\mathbb{Z}$ permutation action on the tensor square of a modular tensor category $\mathcal{C}$. When $\mathcal{C}$ is unpointed we provide formulas for the fusion rules of the gauged category, which non-trivially involves the modular data of $\mathcal{C}$. Our technique highlights the use of generalized Frobenius-Schur indicators. We discuss several examples related to quantum groups at roots of unity. 

\end{abstract}
\maketitle

\section{Introduction}

Gauging is an important concept in the study of 2 dimensional topological phases. It is a procedure for taking a theory with an on-site group of global symmetries and constructing a new theory where the symmetries act locally \cite{1410.4540}. From a categorical perspective, the original topological order is described by a unitary modular tensor category (UMTC), and the symmetry of the system provides a \textit{symmetry enriched topological order}, described by a G-crossed braided extension of $\cC$. The topological order of the gauged theory is described by the equivariantization of the G-crossed extension, and is again a UMTC.


From a purely mathematical point of view, understanding gauging in general is a difficult problem. Gauging is a two-step process. Starting from a categorical action of $G$ on a modular category $\cC$, the first step is to apply the G-crossed braided extension theory of braided fusion categories by following the general recipe of Etingof, Nikshych, and Ostrik \cite{MR2677836}. One begins by finding explicit invertible modules associated to braided auto-equivalences, and then finding appropriate module functors between their relative tensor products. There are two cohomological obstructions associated to these choices which may or may not vanish. If they vanish, then one can construct a $G$-graded extension of $\mathcal{C}$ which, while not braided, is $G$-crossed braided \cite{MR2609644}. In particular this extension carries a categorical action of $G$. For the second step we equivariantize by this $G$-action, to obtain a new non-degenerate braided fusion category. 

 From a mathematical point of view, simply determining whether or not an extension exists can be hard. Often one can establish the existence of an extension abstractly using cohomological arguments but in this situation understanding basic properties of the resulting extension(s), such as the fusion rules and the categorical $G$-action, can be a daunting task. Furthermore, in order to compute the fusion rules of the fully gauged category, it is necessary to understand the tensorators of the $G$-action, not just how it acts on objects. This information can be hard to determine without explicit structure functors and maps in hand. 

 A large class of obvious braided actions on modular tensor categories are the so-called \textit{permutation actions} construcuted from repeated Deligne products of a modular tensor category (see \cite[Definition 1.11.1]{MR3242743} for the definition of the Deligne product of categories). Given a modular category $\mathcal{C}$, the category $\mathcal{C}^{\boxtimes n}$ is also modular, and the symmetric group $S_{n}$ acts on $\cC^{\boxtimes n}$ via braided automorphisms by permuting the tensor factors. Restricting this action to any subgroup of $S_{n}$ is called a \textit{permutation action}. Since these actions always exist, one might expect to be able to say something about gauging them in general. In particular, can one always succeed in finding the G-extension for permutation actions? Can we determine fusion rules (or more generally the modular data) of the resulting gauged category?

 These questions are also related to rational conformal field theory. Given a conformal field theory $A$ (axiomatized by either a rational conformal net \cite{MR1838752} or a $C_{2}$-cofinite vertex operator algebra \cite{MR2140309}), its category of symmetries $\text{Rep}(A)$ is always a modular tensor category and Gannon conjectures that all modular tensor categories arise in this way. Given an action of $G$ by automorphisms on $A$, one can construct the orbifold theory $A^{G}$. In particular, taking the tensor product theory $A^{\otimes n}$ one always has an action by permutation automorphisms. The resulting orbifold is called the \textit{permutation orbifold theory}. In the conformal net axiomatization, M{\"u}ger has shown that $\text{Rep}(A^{G})$ is always a gauging of $\text{Rep}(A)$ by the induced action of $G$ \cite{MR2183964}. Thus the representation category of a permutation orbifold is a permutation gauging of the original category.

There has been some progress on these questions. In \cite{MR2806495}, the authors give a topological construction of the permutation extensions which are a-priori weak fusion categories. For the $\mathbb{Z}/2\mathbb{Z}$ case, the authors explicitly work out the associator and directly show that one obtains a fusion category. In \cite{MR3959559}, the authors show the obstructions vanish for permutation actions of arbitrary non-degenerate braided fusion categories. In both these cases it is very difficult to make any general statements about the structure of the resulting equivariantization. In this paper we study the case of $\mathbb{Z}/2\mathbb{Z}$ permutation actions on $\mathcal{C} \boxtimes \cC$. In the case where $\cC$ has no invertible objects, we give formulas for the fusion rules of the equivariantizations in \textbf{Theorem \ref{fusionequiv}}. We then apply this formula to several examples.

 Surprisingly our arguments for the structure of the extension, and formula for the fusion rules of the gauging rely in a nontrivial way on the spherical structure of the initial category $\mathcal{C}$. In other words, although permutation gauging makes sense for non-degenerate braided fusion categories, most of our formulas (particularly the formula for the fusion rules of the final gauging) only make sense if $\mathcal{C}$ is equipped with a spherical structure. Our formulas agree with the formulas in \cite{MR1606821}, \cite{MR1620424}, \cite{MR1992886} and \cite{MR2116735} obtained in the setting of conformal field theory.
 
 The structure of the paper is as follows. In section $2$ we review the basics of $G$-crossed braided extension theory for braided fusion categories. In section 3 we focus on the case of $\mathbb{Z}/2\mathbb{Z}$ permutation actions of $\mathcal{C}\boxtimes \mathcal{C}$ and give an explicit invertible module which corresponds to the swap action. In section 4 we describe the fusion rules for the extension and also the $\mathbb{Z}/2\mathbb{Z}$-categorical action. In section 5 we apply the formulas of \cite{MR3059899} to determine the fusion rules for the subsequent equivariantization. Finally, in section 6 we present some examples and conjecture on the result of permutation gauging for certain quantum group categories.

\medskip

\section{Preliminaries}

A \textit{fusion category} is a finite semisimple $\mathbb{C}$-linear rigid monoidal category with simple unit \cite{MR2183279}. There are many examples that arise throughout physics and mathematics (see \cite{MR2640343} for an overview). Of particular interest are \textit{braided fusion categories} \cite{MR2609644} which possess natural isomorphisms $\operatorname{br}_{V,W}: V\otimes W\rightarrow W\otimes V$ satisfying certain coherence conditions. 

The \textit{M{\"u}ger center} of a braided fusion category is the full subcategory generated by objects $V$ with $\operatorname{br}_{W,V}\circ \operatorname{br}_{V,W}=\operatorname{id}_{V\otimes W}$ for all $W\in \mathcal{C}$. A braided fusion category is called \textit{non-degenerate} if the M{\"u}ger center is trivial, \textit{i.e.} equivalent to $\Vec$ as a braided fusion category.

A non-degenerate braided fusion category $\cC$ which in addition is equipped with a \textit{spherical structure} (see \cite[Definition 4.7.14]{MR3242743}) is called \textit{modular} \cite{Muger-Oxford}. In this case, there are remarkable numerical invariants associated to the category, namely the S and T matrices, which have fascinating number theoretic properties and deep relationships to topological field theories \cite[Chapter 3]{MR1797619}. Modular categories also naturally appear in applications to conformal field theory since often in physics unitary assumptions are fundamental, which automatically imply the existence of spherical structures. It is an open question whether every fusion category admits a spherical structure and in particular whether every non-degenerate braided fusion category is modular. 

We wish to point out an unfortunate disconnect in the literature. The extension theory developed in \cite{MR2677836} makes references to neither spherical structures (or more fundamentally, pivotal structures), nor unitary structures. Thus starting with a spherical or unitary fusion category and applying the extension theory, there is no guarantee that the result will have a spherical structure or unitary structure respectively. This is widely believed to simply be a matter of no one having yet worked through the details. In particular, adding the word unitary in the appropriate places in the arguments of \cite{MR2677836} fixes the problem, while for spherical structures the question is a bit more subtle. Thus starting from a modular category and applying gauging, as we will do in the sequel, we produce an \emph{a-priori} non-degenerate braided fusion category, not necessarily a modular one. However if we assume our modular category is unitary (as a braided fusion category) and our braided auto-equivalences are $*$-functors, then as in \cite{MR3555361}, we will again obtain a unitary modular category.

We say an object $g$ in a fusion category is \textit{invertible} if we have that $g\otimes g^* \cong \mathbf{1}$. The isomorphism classes of these objects form a group, which we denote $\Inv(\cC)$.

\subsection*{The Picard group, $G$-crossed braided extensions, and gauging}

We briefly recall notions related to the $G$-crossed braided extension theory of braided fusion categories \cite{MR2677836}. 

Let $\mathcal{C}$ be a braided fusion category. Given a $\mathcal{C}$-module $\mathcal{M}$, we can equip it with the structure of a $\cC$-bimodule using the braiding as in \cite[Section 2.8]{MR3107567}. We say a module category is \textit{invertible} if it is invertible as a $\cC$-bimodule \cite[Section 3]{MR2677836}.

The Picard tri-category $\underline{\underline{\operatorname{Pic}}}(\mathcal{C})$ of $\mathcal{C}$ is the tri-category with morphisms
\begin{itemize}
\item 0-morphism: The braided fusion category $\mathcal{C}$, 
\item 1-morphisms: Invertible $\mathcal{C}$-modules,
\item 2-morphisms: Module equivalences,
\item 3-morphisms: Module natural isomorphisms.
\end{itemize}

Composition of $1$-morphisms is given by the relative tensor product of bimodules $\boxtimes_\cC$, while the other compositions are the standard ones (see, for example \cite{MR2678824} for a detailed exposition). This tri-category can be truncated at the top level to form the monoidal category $\underline{\Pic}(\mathcal{C})$. The monoidal category $\underline{\Pic}(\mathcal{C})$ can itself be then truncated to give a group $\Pic(\mathcal{C})$, called the \textit{Picard group} of $\mathcal{C}$. An important application of $\underline{\underline{\Pic}}(\mathcal{C})$ and its associated truncations is in the classification of $G$-crossed braided extensions of $\mathcal{C}$.

A \textit{$G$-crossed braided extension} of $\mathcal{C}$ is a $G$-graded fusion category $\mathcal{D}=\bigoplus_{g\in G} \cC_{g}$ with $\cC_{e}=\cC$, equipped with a categorical $G$-action such that $g(\cC_{h})\subseteq \cC_{ghg^{-1}}$. We also have a \textit{G-braiding}, which is a family of natural isomorphisms 
\[ X \otimes Y \to g(Y)\otimes X,\]
for all $X\in \cC_{g}$, $Y\in \mathcal{D}$, satisfying certain coherence conditions (see \cite[Appendix 5]{MR2674592},\cite[Definition 4.41]{MR2609644}). Note that the categorical $G$-action restricts to a braided $G$-action on the trivial component $\cC$.

 It is shown in \cite[Theorem 7.12]{MR2677836} that for a given group $G$, the $G$-crossed braided extensions of $\mathcal{C}$ are classified by tri-functors $\underline{\underline{G}}\rightarrow \underline{\underline{\Pic}}(\cC)$. Such functors are described by triples $(C,M,A)$, where
\begin{enumerate}\label{obstructions}
\item $C:G \to \Pic(\mathcal{C})$ is a group homomorphism
\item $M$ is a collection of $\cC$-module equivalences \[M_{g,h}: C_g \boxtimes_\cC C_h \to C_{gh},\] such that we have natural isomorphisms of module functors $M_{fg,h}( M_{f,g} \boxtimes_\cC \Id_{C_h}) \cong M_{f,gh}( \Id_{C_f}\boxtimes_\cC M_{g,h})$.
\item $A$ is a specific choice of isomorphisms \[A_{f,g,h}: M_{fg,h}( M_{f,g} \boxtimes_\cC \Id_{C_h}) \to M_{f,gh}( \Id_{C_f}\boxtimes_\cC M_{g,h}),\] satisfying the pentagon equation.
\end{enumerate}

Extension theory \cite{MR2677836} is the process of finding such 3-dimensional data by starting from lower dimensional data and attempting to lift it to higher dimensional data. At each stage there is an obstruction to lifting, given by a certain group cohomology class. We briefly explain this process starting from the bottom up.

Suppose we have a homomorphism $C:G\rightarrow \Pic(\cC)$, there is an obstruction to lifting this to a monoidal functor $\underline{G}\rightarrow \underline{\Pic}(\cC)$, representented by a cohomology class $o_{3}(C)\in H^{3}(G, \Inv(\cC))$. To compute $o_{3}(C)$, one first chooses any collection of $\cC$-module equivalences $M=\{M_{g,h}: \cC_{g}\boxtimes_{\cC} \cC_{h}\cong \cC_{gh}\}_{g,h\in G}$. Then one constructs $T(C,M)\in Z^{3}(G, \Inv(\cC))$ \cite[Equation (53)]{MR2677836}, which measures how far the choices $M_{f,g}$ are from satisfying $M_{fg,h}( M_{f,g} \boxtimes_\cC \Id_{C_h}) \cong M_{f,gh}( \Id_{C_f}\boxtimes_\cC M_{g,h})$. Then one defines $o_{3}(C):=[T(C,M]\in H^{3}(G, \Inv(\cC))$, which only depends on $C$. There exists a lifting of $C$ if and only if $o_{3}(C)$ is trivial. Furthermore, if the obstruction vanishes, equivalence classes of liftings form a torsor over $H^{2}(G, \Inv(\cC))$.

Similarly, given a monoidal functor $(C,M):\underline{G}\rightarrow \underline{Pic}(\cC)$, there is an obstruction to lifting this to a tri-functor $\underline{\underline{G}}\rightarrow \underline{\underline{\Pic}}(\cC)$. For any choice of $\cC$-module functor natural isomorphisms $A=\{A_{f,g,h}: M_{fg,h}( M_{f,g} \boxtimes_\cC \Id_{C_h})\cong M_{f,gh}( \Id_{C_f}\boxtimes_\cC M_{g,h})\}_{f,g,h\in G}$ one defines $v(C,M,A)\in Z^{4}(G,\mathbb{C}^{\times})$ \cite[Equations (53) and (58)]{MR2677836}, which measures how far the collection $A_{f,g,h}$ are from satisfying the pentagon equations. The cohomology class $o_{4}(C,M):=[v(C,M,A)]$ only depends on the monoidal functor $(C,M)$. There exists a lifting of $(C,M)$ to a tri-functor if and only $o_{4}(C,M)$ vanishes. In this case, the possible liftings form a torsor over $H^{3}(G, \mathbb{C}^{\times})$. For details see \cite[Section 8]{MR2677836}.

Given a $G$-crossed braided extension of $\mathcal{C}$, one can equivariantize by the $G$-action to get a braided fusion category. When the $0$-graded component of the G-crossed braided category is modular (or more generally, non-degenerate), the result of equivariantizing also modular (non-degenerate), see \cite[Proposition 4.56]{MR2609644}. The process of starting with a modular category $\mathcal{C}$, constructing a $G$-crossed braided extension, and then equivariantizating to obtain a new modular category is known as \textit{gauging}. Gauging has been studied from the physics point of view in \cite{1410.4540}, where it corresponds to orbifolding. An orbifold of a conformal field theory is the new conformal field theory obtained by taking the fixed points under a group action. From the mathematical point of veiw, gauging has been studied in \cite{MR3555361}.

We define $\eqbr(\mathcal{C})$ to be the monoidal category whose objects are braided auto-equivalences of $\cC$, and whose morphisms are monoidal natural isomorphisms. It is shown in \cite[Theorem 5.2]{MR2677836} that for a non-degenerate category $\mathcal{C}$ , there is a monoidal equivalence 
\[\underline{\Pic}(\mathcal{C}) \to \eqbr(\mathcal{C}). \]
We briefly discuss one part of this equivalence which will be useful for us, namely how to construct a braided auto-equivalence of $\mathcal{C}$ from an invertible module. 

\begin{rem} When working in the $2$-category of categories, functors, and natural transformations, we use the $\circ$ notation to denote the horizontal composition of natural transformations (compatible with $\circ$ for composition of functors), while $\cdot $ denotes ordinary vertical composition of natural transformations.
\end{rem}

Given a braided fusion category $\mathcal{C}$ and an invertible module $\mathcal{M}$, we can construct two monoidal functors $\alpha^{\pm}_{\mathcal M}: \mathcal C \to \operatorname{Fun}_{\mathcal{C}}(\mathcal{M},\mathcal{M})$, which are
\[ m \mapsto X\triangleright m \]
However, the module functor structure maps differ for $\alpha^+(X)$ and $\alpha^-(X)$. They are, respectively, given by
\begin{align*}
L^+_{Y,m} &:= {l^{\mathcal{M}}_{X,Y,m}}^{-1}\operatorname{br}_{Y,X}l^{\mathcal{M}}_{Y,X,m}: Y\triangleright \alpha^+(X)[m] \to \alpha^+(X)[Y\triangleright m], \\ 
L^-_{Y,m} &:= {l^{\mathcal{M}}_{X,Y,m}}^{-1}\operatorname{br}^{-1}_{X,Y}l^{\mathcal{M}}_{Y,X,m}: Y\triangleright \alpha^-(X)[m] \to \alpha^-(X)[Y\triangleright m]. 
\end{align*}

Here, $l^{\mathcal{M}}_{Y,X,m}:Y\triangleright(X\triangle m)\rightarrow (Y\otimes X)\triangleright m$ are the module structure natural isomorphisms for $\mathcal{M}$. We use the module natural isomorphisms as the tensorators for both $\alpha^+$ and $\alpha^-$
\[l^{\mathcal{M}}_{X,Y,-} : \alpha^{\pm}(X) \circ \alpha^{\pm}(Y) \to \alpha^{\pm}(X\otimes Y) .\]

When $\mathcal{C}$ is non-degenerate, these monoidal functors are both equivalences, and the equivalence ${\alpha_{\mathcal{M}}^-}^{-1}\circ \alpha^+_{\mathcal{M}} : \cC \to \cC$ is braided.

\subsection*{Vanishing obstructions for permutation actions}

Recall that $\underline{G}$ is the monoidal category whose objects are elements of $G$ and morphisms are identities.

\begin{defn}
A \textit{braided categorical action} is a monoidal functor $\underline{G}\rightarrow \eqbr(\mathcal{C})$.
\end{defn}

When $\mathcal{C}$ is non-degenerate, composing a braided categorical action with the equivalence $\eqbr(\mathcal{C})\rightarrow \underline{\Pic}(\mathcal{C})$ gives a monoidal functor $\underline{G}\rightarrow \underline{\Pic}(\mathcal{C})$. This gives invertible $\mathcal{C}$-modules $\{C_{g}\}_{g\in G}$. The tensorator of this functor provides explicit choices of module equivalences $M_{f,g}:C_{f}\boxtimes_{\cC} C_{g}\rightarrow C_{fg}$, and the condition that this functor is monoidal implies that $T(C,M)$ is trivial.

A large class of braided categorical actions are given by \textit{permutation actions}, defined as follows:

\begin{defn} Let $\mathcal{C}$ be a braided fusion category and $G$ a subgroup of a permutation group $S_{n}$. The associated \textit{permutation action} is the braided categorical action of $G$ on the Deligne tensor power $\mathcal{C}^{\boxtimes n}$, where $G$ acts by permuting the tensor factors, and the tensorators are identities.

\end{defn}

Given a permutation action $G$ on $\cC$, it is natural to try construct a corresponding $G$-crossed braided extension of $\cC$. While the above discussion shows that we can make coherent choices of $C$ and $M$, it is not clear if we can make a coherent choice of $A$ to give a $G$-crossed braided category. It was shown in \cite{MR3959559} that it is always possible to make a coherent such choice of $A$, using involved cohomology arguments. To help keep this paper self contained we prove the following Lemma, which gives a straight forward proof of the existence, and classification of choices of $A$.

\begin{lem}\label{lem:existence} There are precisely $N$ distinct $\mathbb{Z}/N\mathbb{Z}$-crossed braided extensions of $\mathcal{C}^{\boxtimes N}$ corresponding to the cyclic permutation action, obtained from each other by twisting the associator with representatives of elements in $H^{3}(\mathbb{Z}/N\mathbb{Z}, \mathbb{C}^{\times})$, i.e the $\mathbb{Z}/N\mathbb{Z}$-crossed braided extensions of $\mathcal{C}^{\boxtimes N}$ form a torsor over $H^{3}(\mathbb{Z}/N\mathbb{Z}, \mathbb{C}^{\times})$.
\end{lem}

\begin{proof}
As the permutation action of $\mathbb{Z}/N\mathbb{Z}$ on $\cC^{\boxtimes N}$ is categorical, we have that there exists some collection of module equivalences $M$ such that the first obstruction to constructing a $\mathbb{Z}/N\mathbb{Z}$-crossed braided extension vanishes. As this obstruction vanishes, there exists some collection of module natural isomorphisms $A$. As the group $H^4(\mathbb{Z}/N\mathbb{Z}, \mathbb{C}^\times)$ is trivial, we can rescale the collection $A_{f,g,h}$ so that the second obstruction to constructing a $\mathbb{Z}/N\mathbb{Z}$-crossed braided extension vanishes too. Hence, there exists a $\mathbb{Z}/N\mathbb{Z}$-crossed braided extension of $\mathcal{C}^{\boxtimes N}$ corresponding to the cyclic permutation action.

In general from \cite[Section 8]{MR2677836} we know that equivalence classes of collections $M$ such that the obstruction $T(C,M)$ vanish form a torsor over $H^{2}(G, \Inv(\mathcal{C}))$. In our cyclic permutation case the relevant group is $H^{2}(\mathbb{Z}/N\mathbb{Z}, \Inv(\mathcal{C})^{N})$, where $\mathbb{Z}/N\mathbb{Z}$ acts on $\Inv(\mathcal{C})^{N}$ by cyclicly permuting the factors. It is well-known from the theory of group cohomology that $H^{2}(\mathbb{Z}/N\mathbb{Z}, M)=M^{\mathbb{Z}/N\mathbb{Z}}/D(M)$, where $D=\prod_{g\in\mathbb{Z}/N\mathbb{Z}} g\in End(M)$. In our case $(\Inv(\mathcal{C})^{N})^{\mathbb{Z}/N\mathbb{Z}}$ is exactly the diagonal elements in $\Inv(\mathcal{C})^{N}$. But every diagonal element $(X, \dots, X)\in \Inv(\mathcal{C})^{N}$ is of the form $D(X,1,\dots,1)$. Thus $H^{2}(\mathbb{Z}/N\mathbb{Z}, \Inv(\cC)^N) = \{e\}$. This implies that there is a unique collection of module equivalences $M_{f,g}$ giving rise to a $\mathbb{Z}/N\mathbb{Z}$-crossed braided extension.

It is a well known fact that $H^3( \mathbb{Z}/N\mathbb{Z}, \mathbb{C}^\times) = \mathbb{Z}/N\mathbb{Z}$. Thus there are $N$ different collections of module natural isomorphisms $A$ giving rise to $\mathbb{Z}/N\mathbb{Z}$-crossed braided extensions. 
\end{proof}

While the above categories are distinct as $G$-crossed braided extensions of $\mathcal{C}$ they may be equivalent as monoidal categories (see \cite{1711.00645} for details). 

While the above Lemma shows the abstract existence of these $\mathbb{Z}/N\mathbb{Z}$-crossed braided extensions, the proof is very much non-constructive. For the remainder of this paper we will restrict our attention to the $\mathbb{Z}/2\mathbb{Z}$-case and give an explicit construction of the $\mathbb{Z}/2\mathbb{Z}$-crossed braided extension. Furthermore, we compute the fusion rules for the equivariantization by the $\mathbb{Z}/2\mathbb{Z}$-action to complete the gauging process. While most of our proofs only hold in the case that our initial category has no non-trivial invertible objects, we conjecture that they are still valid in the general case.

\section{The invertible module coming from the swap action of $\mathcal{C}\boxtimes \mathcal{C}$}\label{sec:module}

Let $\cC$ be a modular tensor category. We will explicitly describe the invertible module category of $\cC \boxtimes \cC$ associated to the swap auto-equivalence. This module category was constructed and studied in \cite{MR2673719}, but we use a slightly different (but equivalent) version.

\begin{rem}
We will write $\sigma$ for the swap auto-equivalence of $\cC \boxtimes \cC$.
\end{rem}
 We define $\widehat{\cC}:=\cC$ as a linear category and we equip it with the structure of a left module category over $\cC\boxtimes \cC$ as follows.

	The action is given by 
\begin{equation*}
X \boxtimes Y \triangleright \widehat{M} := \widehat{X\otimes M \otimes Y},
\end{equation*}
with module structure morphisms
\begin{equation*}
l_{X\boxtimes Y, Z\boxtimes W,\widehat{M}}:= \id_{X\otimes Z \otimes M} \otimes \operatorname{br}_{W,Y} : X\boxtimes Y \triangleright Z\boxtimes W \triangleright \widehat{M} \to (X\otimes Z) \boxtimes (Y\otimes W)\triangleright \widehat{M},
\end{equation*}
which we draw graphically as
\[ 
 l_{X\boxtimes Y, Z\boxtimes W,\widehat{M}} =\raisebox{-.5\height}{ \includegraphics[scale = .3]{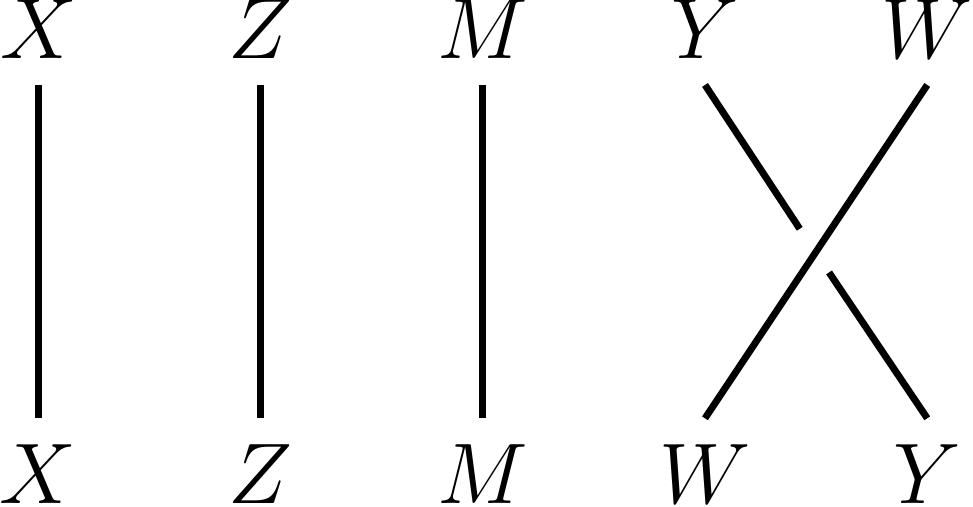}}.
\]

\begin{lem}\label{lem:mod}
The image of $\widehat{\cC}$ under the equivalence $\underline{\Pic}(\cC\boxtimes \cC) \to \eqbr(\cC \boxtimes \cC)$ is $\sigma$.
\end{lem}
\begin{proof}
The statement of this lemma is equivalent to showing that the equivalences $\alpha^+_{\widehat{\cC}}$ and  $\alpha^-_{\widehat{\cC}}\circ \sigma$ are monoidally isomorphic.

Let us look at $\alpha^+_{\widehat{\cC}}$ and $\alpha^-_{\widehat{\cC}}$ in more detail. Both are monoidal functors $\cC\boxtimes \cC \to \operatorname{Fun}_{\cC\boxtimes \cC}(\widehat{\cC},\widehat{\cC})$. For a given object $X\boxtimes Y \in \cC\boxtimes \cC$, both $\alpha^+_{\widehat{\cC}}$ and $\alpha^-_{\widehat{\cC}}$ send this object to the same endofunctor $X\boxtimes Y \triangleright \widehat{?} = \widehat{X\otimes ? \otimes Y}$. However, the module functor structure maps for $\alpha^+_{\widehat{\cC}}(X\boxtimes Y)$ and $\alpha^-_{\widehat{\cC}}(X\boxtimes Y)$ are slightly different. They are respectively given by:

\begin{align*}
L^+_{X\boxtimes Y, Z\boxtimes W, \widehat{M}} &:= l_{X\boxtimes Y, Z\boxtimes W, \widehat{M}}^{-1} \operatorname{br}_{X\boxtimes Y, Z\boxtimes W}\triangleright \id_{ \widehat{M}} l_{Z\boxtimes W, X\boxtimes Y, \widehat{M}},\\
L^-_{X\boxtimes Y, Z\boxtimes W, \widehat{M}} &:= l_{X\boxtimes Y, Z\boxtimes W, \widehat{M}}^{-1} \operatorname{br}^{-1}_{Z\boxtimes W, X\boxtimes Y}\triangleright \id_{ \widehat{M}} l_{Z\boxtimes W, X\boxtimes Y, \widehat{M}}.
\end{align*}

Graphically, we draw these structure maps as:
\begin{center}
$L^+_{X\boxtimes Y, Z\boxtimes W, \widehat{M}} = \raisebox{-.5\height}{ \includegraphics[scale = .3]{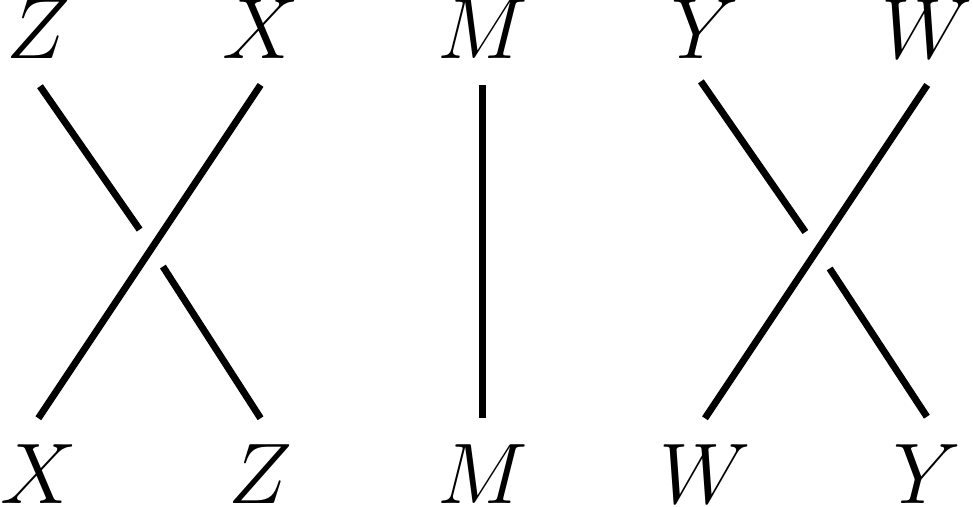}}, \quad L^-_{X\boxtimes Y, Z\boxtimes W, \widehat{M}} = \raisebox{-.5\height}{ \includegraphics[scale = .3]{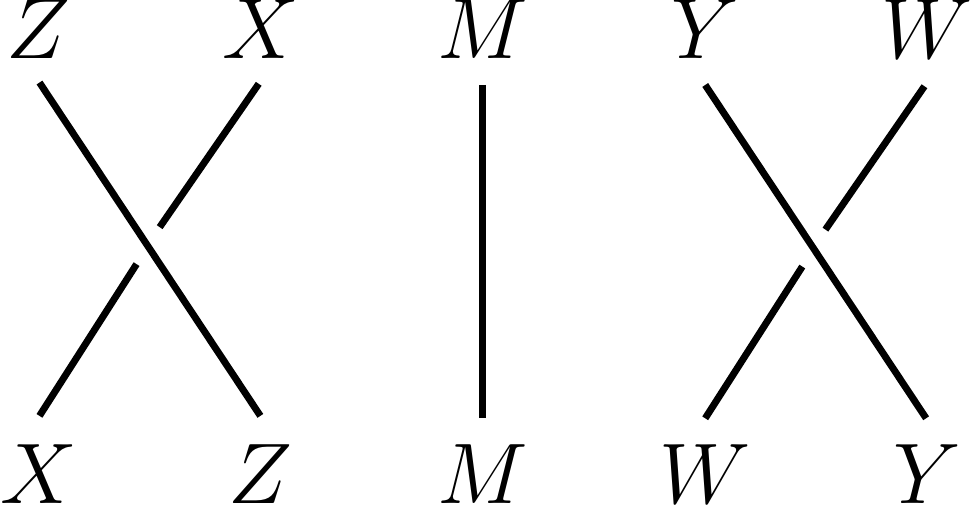}}.$
\end{center}

To complete our proof we have to give a monoidal natural isomorphism $\eta: \alpha^+_{\widehat{\cC}} \to \alpha^-_{\widehat{\cC}} \circ \sigma$. For each $X\boxtimes Y \in \cC\boxtimes \cC$, the component $\eta^{X\boxtimes Y}$ will be a module natural isomorphism $\alpha^+_{\widehat{\cC}}(X\boxtimes Y) \to \alpha^-_{\widehat{\cC}}(Y\boxtimes X)$. For each $\widehat{M} \in \widehat{\cC}$, we define the components of the module natural isomorphism $\eta^{X\boxtimes Y}$ as follows:
\begin{center}
$\eta_{\widehat{M}}^{X\boxtimes Y}$ := \raisebox{-.5\height}{ \includegraphics[scale = .4]{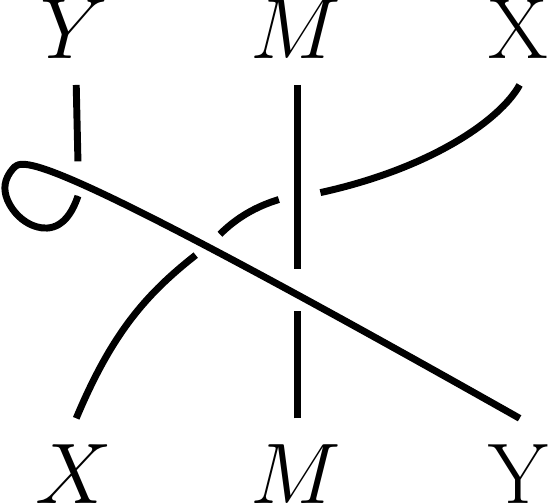}}.
\end{center}
Note that this natural isomorphism involves a twist of the object $Y$, which requires the pivotal strucure of our modular tensor category $\cC$.

It is a straightforward braid isotopy to show that the following diagram commutes:
$$\begin{CD}
Z\boxtimes W \triangleright\alpha^+_{\widehat{\cC}}(X\boxtimes Y) ( \widehat{M}) @> L^+_{Z\boxtimes W, X\boxtimes Y, \widehat{M}} >> \alpha^+_{\widehat{\cC}}(X\boxtimes Y) ( Z\boxtimes W \triangleright \widehat{M})\\
@V \id_{Z\boxtimes W}\triangleright \eta^{X\boxtimes Y}_{\widehat{M}} VV @VV \eta^{X\boxtimes Y}_{Z\boxtimes W \triangleright \widehat{M}} V \\
Z\boxtimes W \triangleright \alpha^-_{\widehat{\cC}}(Y\boxtimes X) ( \widehat{M}) @> L^-_{Z\boxtimes W, X\boxtimes Y, \widehat{M}} >> \alpha^-_{\widehat{\cC}}(Y\boxtimes X) (Z\boxtimes W \triangleright \widehat{M}).
\end{CD}$$

Thus $\eta^{X\boxtimes Y}$ is a module natural isomorphism $\alpha^+_{\widehat{\cC}}(X\boxtimes Y) \to \alpha^-_{\widehat{\cC}}(Y\boxtimes X)$. 

Finally, we have to show that $\eta$ (whose components are the natural isomorphisms $\eta^{X\boxtimes Y}$'s) is a monoidal natural isomorphism from $\alpha^+_{\widehat{\cC}} \to \alpha^-_{\widehat{\cC}} \circ \sigma$. That is we have to show the following diagram commutes:
$$\begin{CD}
\alpha^+_{\widehat{\cC}}(X_1\boxtimes Y_1)\circ \alpha^+_{\widehat{\cC}}(X_2\boxtimes Y_2) @> l_{X_1\boxtimes Y_1, X_2\boxtimes Y_2,-} >>\alpha^+_{\widehat{\cC}}(X_1X_2\boxtimes Y_1Y_2) \\
@V \eta^{X_1\boxtimes Y_1}\circ \eta^{X_2\boxtimes Y_2} VV @VV \eta^{X_1X_2\boxtimes Y_1Y_2} V \\
\alpha^-_{\widehat{\cC}}(Y_1\boxtimes X_1)\circ \alpha^-_{\widehat{\cC}}(Y_2\boxtimes X_2) @> l_{Y_1\boxtimes X_1, Y_2\boxtimes X_2,-} >> \alpha^-_{\widehat{\cC}}(Y_1Y_2\boxtimes X_1X_2).
\end{CD}$$
This equation is an equality of natural isomorphisms. Thus we have to show that each component is equal. Let $\widehat{M} \in \widehat{\cC}$, then

\[(\eta^{X_1X_2\boxtimes Y_1Y_2}\cdot l_{X_1\boxtimes Y_1, X_2 \boxtimes Y_2,-})_{\widehat{M}} =\raisebox{-.5\height}{ \includegraphics[scale = .4]{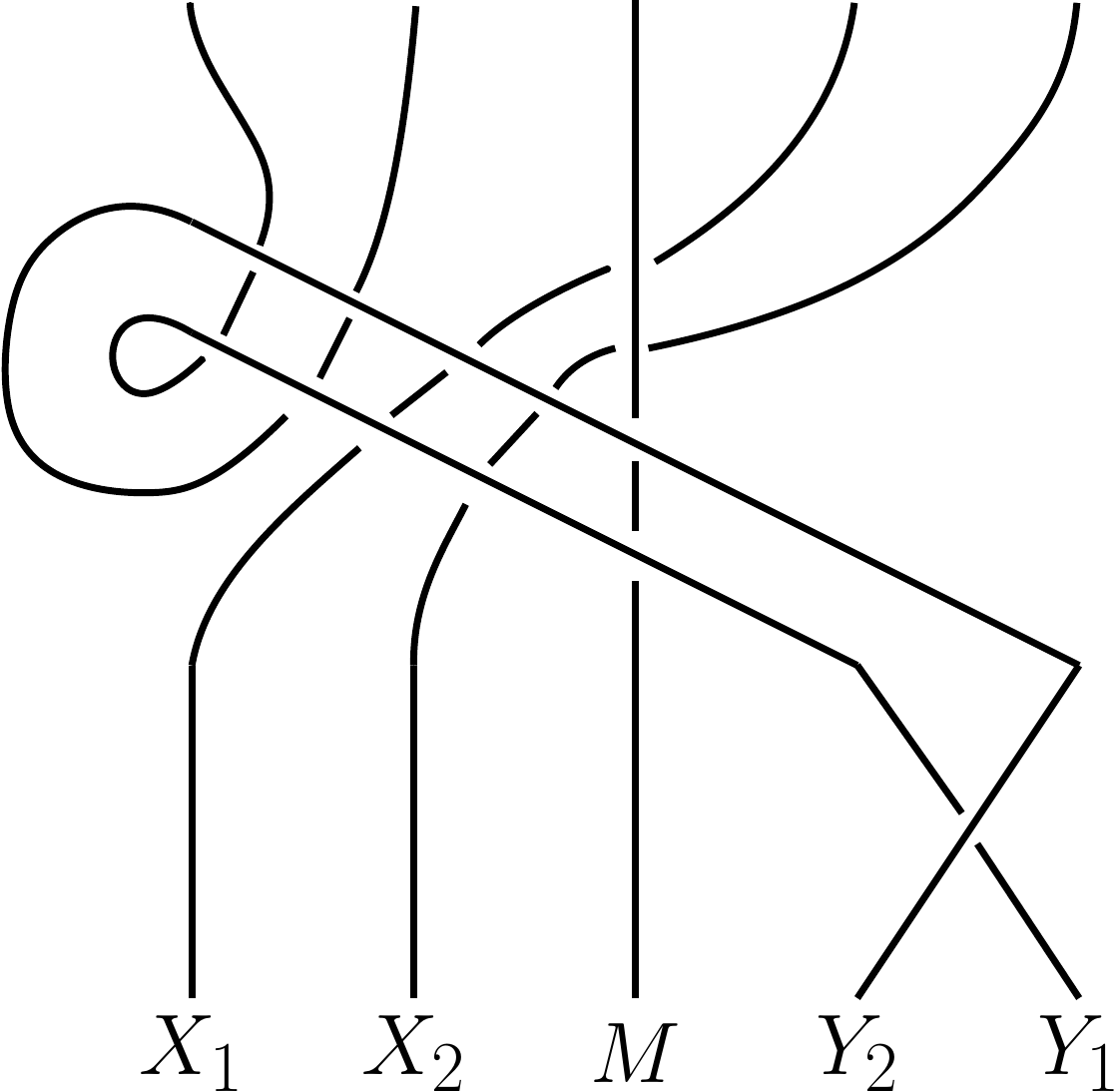}},\]

and

\[(l_{X_1\boxtimes Y_1 X_2\boxtimes Y_2,-} \cdot \eta^{X_1\boxtimes Y_1}\circ \eta^{X_2\boxtimes Y_2})_{\widehat{M}} = \raisebox{-.5\height}{ \includegraphics[scale = .4]{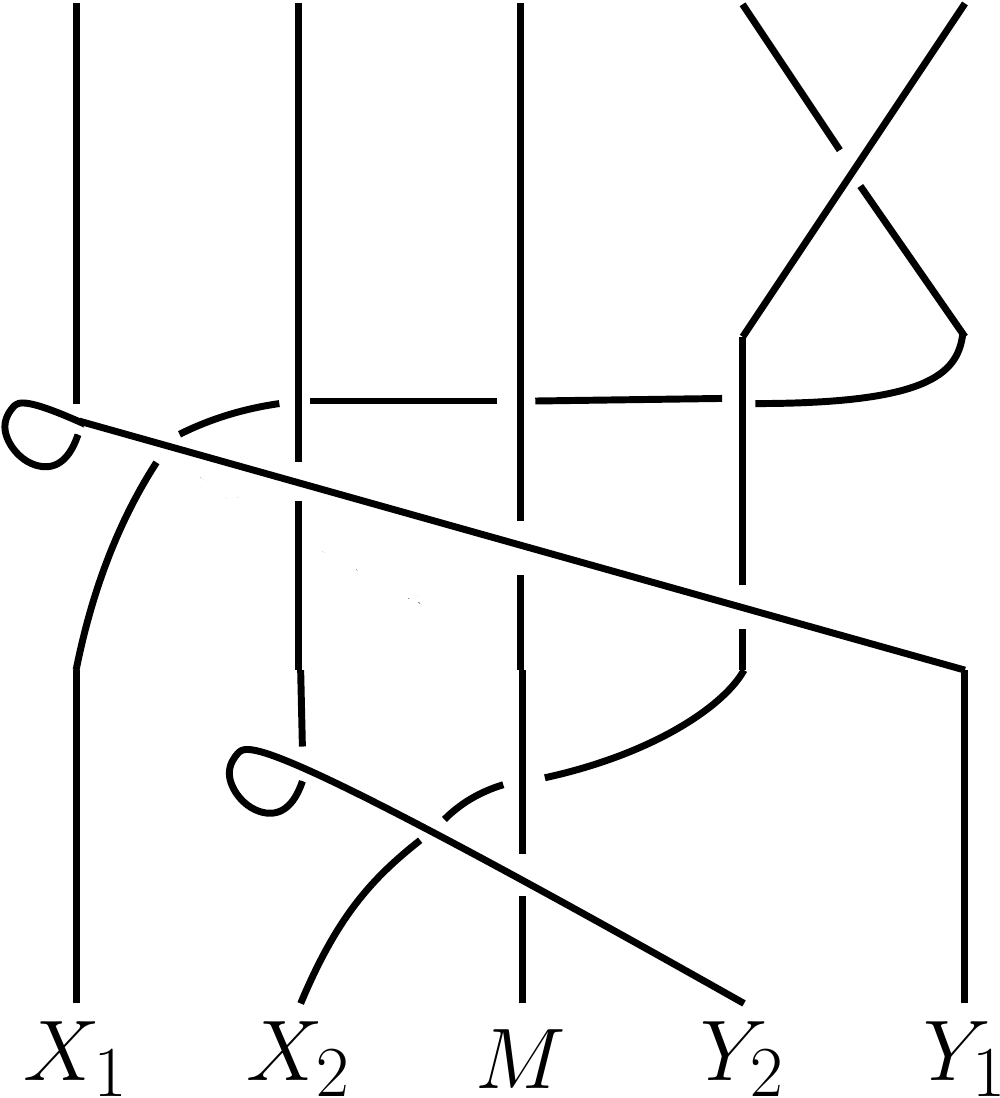}}.\]
Again this is a straightforward isotopy.

Hence $\eta: \alpha^+_{\widehat{\cC}} \to \alpha^-_{\widehat{\cC}}\circ \sigma$ is a monoidal natural isomorphism, which proves the statement of the Lemma.
\end{proof}


\section{$\mathbb{Z}/ 2\mathbb{Z}$-crossed braided extension of $\mathcal{C}\boxtimes\mathcal{C}$ by $\widehat{\mathcal{C}}$ and gauging}

By Lemma~\ref{lem:existence} we know there exist precisely two $\mathbb{Z}/2\mathbb{Z}$-crossed braided extensions of $\mathcal{C}\boxtimes \mathcal{C}$ corresponding to the swap action. Since these two extensions are related by a $3$-cocycle twist, they have the same fusion rules. We have just shown in Lemma~\ref{lem:mod} that the non-trivial graded piece of both of these extensions is $\widehat{\cC}$ (as a module category). We now aim to determine these fusion rules and the $\mathbb{Z}/2\mathbb{Z}$-action on the extensions $(\cC \boxtimes \cC) \oplus \widehat{\cC}$. Fortunately, one of the extensions was explicitly constructed in \cite[Section 4]{MR2806495}, and they compute the fusion rules. We give here a short argument which only applies to the case when $\mathcal{C}$ has no invertible objects.

\begin{prop}\cite[Formula 58]{MR2806495}\label{fusionextension} Let $\mathcal{C}$ be a modular category. Let $\mathcal{D}:=(\mathcal{C}\boxtimes \mathcal{C})\oplus \widehat{\mathcal{C}}$ be either of the extensions constructed in Lemma~\ref{lem:existence} corresponding to the swap action on $\mathcal{C}\boxtimes \mathcal{C}$. Then $\dim(\mathcal{D}(\widehat{X}\otimes \widehat{Y}, Z\boxtimes W))=\dim(\mathcal{C}(XY, ZW))$.
\end{prop}

We adopt the notation $\mathcal C (X, Y)$ to denote the space of morphisms in $\mathcal C$ from $X$ to $Y$. Also, we sometimes just use juxtaposition to denote the tensor product of objects in $\mathcal C$.

\begin{proof}[Proof of Proposition~\ref{fusionextension}, assuming $\cC$ has no non-trivial invertible objects]
Taking left (or right) duals of $\mathcal D$ gives a monoidal functor $\mathcal{D}\rightarrow \mathcal{D}^{\mop}$, which since $\cC\boxtimes\cC$ itself is pivotal restricts to a $\cC\boxtimes\cC$-module equivalence $L^{*}: \widehat{\cC} \to \widehat{\mathcal{C}}^{\op}$. It is straightforward to see that the pivotal structure on $\cC$ allows us to extend the functor $L_{*}(\widehat{X}):=\widehat{X^{*}} : \widehat{\mathcal{C}}\rightarrow \widehat{\mathcal{C}}^{\op}$ to a $\mathcal{C}\boxtimes\mathcal{C}$-module equivalence. As $\widehat{\mathcal{C}}$ is an invertible module, the equivalence classes of $\mathcal C\boxtimes \mathcal C$-bimodule equivalences from $\widehat{\mathcal{C}}$ to $\widehat{\mathcal{C}}^{\op}$ are a torsor over the set of invertible objects in $\cC$. By hypothesis there are no non-trivial invertible elements of $\cC$, thus $L_{*}:\widehat{\mathcal{C}}\rightarrow \widehat{\mathcal{C}}^{\op}$ is the unique $\mathcal{C}\boxtimes \mathcal{C}$-module equivalence. This implies $L_{*}\cong L^{*}$ as $\mathcal{C}\boxtimes \mathcal{C}$-module functors, and in particular $^{*}\widehat{X}\cong \widehat{X^{*}}$. Therefore 

$$\dim (\mathcal{D}(\widehat{X}\otimes \widehat{Y}, Z\boxtimes W))=\dim(\mathcal{D}(\widehat{Y}, \widehat{X}^{*} (Z\boxtimes W)))$$
$$=\dim(\mathcal{D}(\widehat{Y}, \widehat{ZX^{*}W}))=\dim(\mathcal{C}(Y, ZX^{*}W))=\dim(\mathcal{C}(XY, ZW)),$$

where in the last step we have used the braiding in $\mathcal{C}$.

\end{proof}

We already know the fusion rules of the subcategory $\mathcal{C}\boxtimes \mathcal{C}$ (the trivial component of the grading) and the $\mathcal{C}\boxtimes \mathcal{C}$-action on $\widehat{\mathcal{C}}$ gives us the rules to tensoring objects of $\mathcal{C}\boxtimes \mathcal{C}$ with objects of $\widehat{\mathcal{C}}$. Then, Propositon \ref{fusionextension} completes the description of the fusion rules for the $\mathbb{Z}/2\mathbb{Z}$-crossed braided extension corresponding to the swap action on $\mathcal{C}\boxtimes \mathcal{C}$. This argument is a major shortcut, which takes advantage of the general nature of the extension theory.

We also would like to know how the non-trivial element of $\mathbb{Z}/2\mathbb{Z}$ acts on this extension, since this would allow us to work out the fusion rules for the equivariantization. As the $\mathbb{Z}/2\mathbb{Z}$-crossed braided extension $(\mathcal{C}\boxtimes \mathcal{C})\oplus \widehat{\mathcal{C}}$ was constructed from the swap auto-equivalence of $\cC \boxtimes \cC$, it follows that the action on the trivial piece of the extension is just the swap auto-equivalence. The action on the non-trivial piece $\widehat{\cC}$ is more difficult to compute. While we aren't able to describe the full tensorator for the non-trivial auto-equivalence, we can compute a small part of it in the case there are no invertible objects. We will see in the next section that this small part is enough to compute the fusion rules for the equivariantization with respect to this action. We also note that in \cite{MR2806495}, they give a formula for the braiding, but there seems to be a certain morphism $\sigma_{M}$ \cite[Formula 107]{MR2806495} which appears in their construction with no formula in terms of $\mathcal{C}$. Naive interpreations as square roots of twist in $\cC$ seem to be inconsistent.

\begin{prop}\label{prop:action} Let $\mathcal{C}$ be a modular category with no non-trivial invertible objects, and $( \mathcal{C}\boxtimes \mathcal{C} )\oplus \widehat{\mathcal{C}}$ either of the $\mathbb{Z}/2\mathbb{Z}$-crossed braided extension constructed in Lemma~\ref{lem:existence}. Let $\Phi$ be the non-trivial auto-equivalence of this $\mathbb{Z}/2\mathbb{Z}$-crossed braided extension. Then $\Phi |_{\widehat{\mathcal{C}}}$ is isomorphic to $\Id_{\widehat{\mathcal{C}}}$ as a functor. A subset of the full tensorator for $\Phi$ is given by:
\[ \mu_{X\boxtimes Y, \widehat{M}} := \raisebox{-.5\height}{ \includegraphics[scale = .4]{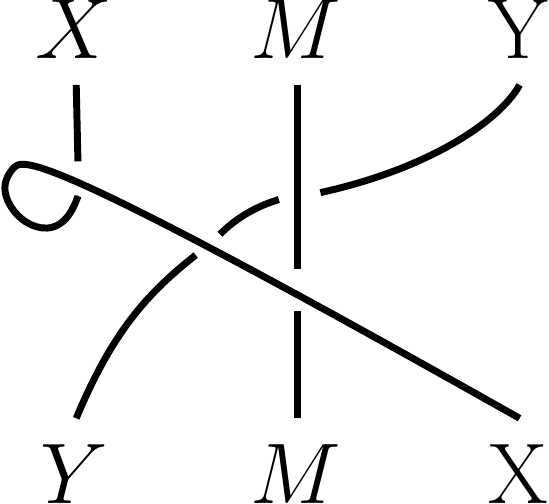}} : \Phi(X\boxtimes Y) \otimes \Phi(\widehat{M} )\to \Phi (X\boxtimes Y \otimes \widehat{M}). \]

Furthermore, a monoidal natural isomorphism $\eta$ from $\Phi^2$ to the identity is given by:
\[ \eta_{X\boxtimes Y} := \id_{X\boxtimes Y} \qquad \text{and} \qquad \eta_{\widehat{X}} =\raisebox{-.5\height}{ \includegraphics[scale = .4]{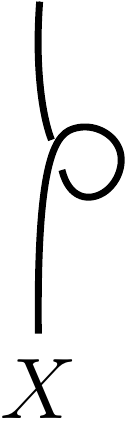}}.\]
\end{prop}

\begin{proof} 
From results of \cite[Section 5]{1711.00645}, $\Phi |_{\widehat{\mathcal{C}}}$ has the structure of a $\sigma$-twisted module auto-equivalence of $\widehat{\mathcal{C}}$ (as defined in Definition 5.1 of the same paper). It is straightforward to verify that the functor $\Id_{\widehat{\mathcal{C}}}$, with structure morphisms 
\[ \raisebox{-.5\height}{ \includegraphics[scale = .4]{YangBaxterator2}} : \sigma(X\boxtimes Y) \otimes \Id_{\widehat{\mathcal{C}}}(\widehat{M} )\to \Id_{\widehat{\mathcal{C}}}(X\boxtimes Y \otimes \widehat{M}) \]
 is a $\sigma$-twisted module auto-equivalence of $\widehat{\mathcal{C}}$. Furthermore, as $\sigma$-twisted module auto-equivalences of $\widehat{\mathcal{C}}$ form a torsor over the invertible elements of $\widehat{\mathcal{C}}$, we must have that the functor given in the statement is the only such one. Thus $\Phi |_{\widehat{\cC}}$ has to be the described functor.

Let $\tau |_{\widehat{\cC}}$ be the restriction of the tensorator of $\Phi$ to $\widehat{\mathcal{C}}$. While we do not compute $\tau |_{\widehat{\cC}}$ directly, we note that as $Z^2(\mathbb{Z}/2\mathbb{Z} , \mathbb{C}^\times) = B^2(\mathbb{Z}/2\mathbb{Z} , \mathbb{C}^\times) =\mathbb{C}^\times$. Then, we can rescale $\tau |_{\widehat{\cC}}$ by any non-zero complex number, without changing the isomorphism class of $\tau$. Details on why we can preform this rescaling can be found in the proof of \cite[Lemma 8.2]{1711.00645}.

As the monoidal auto-equivalence $\Phi$ came from a $\mathbb{Z}/2\mathbb{Z}$-categorical action, we know there exists a monoidal natural isomorphism $\eta: \Phi^2 \to \Id_{\mathcal{C}\boxtimes \mathcal{C}\oplus \widehat{\mathcal{C}}}$. It is straightforward to verify that $\eta_{X\boxtimes Y} = \id_{X\boxtimes Y}$ satisfies the conditions to be a monoidal natural isomorphism $\sigma\circ\sigma \to \Id_{\cC \boxtimes \cC}$.

We can restrict $\eta$ to $\widehat{\cC}$ to get a module functor natural isomorphism:
\[ \eta |_{\widehat{\cC}} : \Phi^2 |_{\widehat{\mathcal{C}}} \to \Id_{\widehat{\cC}}.\]

The isomorphisms $\alpha \raisebox{-.5\height}{ \includegraphics[scale = .4]{X_twist}} : \Phi^2 |_{\widehat{\mathcal{C}}} \to \Id_{\widehat{\cC}}$, for $\alpha \in \mathbb{C}^\times$, have the structure of module functor natural isomorphisms $\Phi^2 |_{\widehat{\mathcal{C}}} \to \Id_{\widehat{\cC}}$. As module functor natural isomorphisms $\Phi^2 |_{\widehat{\mathcal{C}}} \to \Id_{\widehat{\cC}}$ form a torsor over $\mathbb{C}^\times$, we must have that $\eta |_{\widehat{\cC}} = \alpha \raisebox{-.5\height}{ \includegraphics[scale = .4]{X_twist}}$ for some (still undetermined) scalar $\alpha$.

To extend the module functor natural isomorphism $ \alpha \raisebox{-.5\height}{ \includegraphics[scale = .4]{X_twist}}$ to a monoidal natural isomorphism $\Phi^2 \to \Id$, we need the following diagram to commute:
$$\begin{CD}
 \Phi^2(\widehat{X}) \otimes \Phi^2(\widehat{Y}) @> \alpha \raisebox{-.5\height}{ \includegraphics[scale = .4]{X_twist}} \otimes \alpha \raisebox{-.5\height}{ \includegraphics[scale = .4]{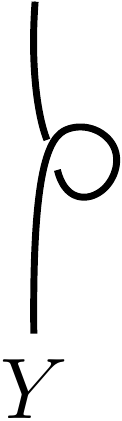}} >>\widehat{X} \otimes \widehat{Y} \\
@V \tau_{\widehat{X},\widehat{Y} }\cdot\tau_{\widehat{X},\widehat{Y} } VV @VV \id_{\widehat{X}\otimes \widehat{Y} } V \\
\Phi^2(\widehat{X}\otimes \widehat{Y}) @>\eta_{\widehat{X}\otimes \widehat{Y}} >> \widehat{X}\otimes \widehat{Y}.
\end{CD}$$
 We know the monoidal natural isomorphism $\eta$ exists, therefore there must exist some non-zero $\alpha$ making the above diagram commute. We can then rescale $\tau |_{\widehat{\cC}}$ by $\alpha$ to arrange that $\eta_{\widehat{X}} = \raisebox{-.5\height}{ \includegraphics[scale = .4]{X_twist}}$.

\end{proof}

\section{Equivariatization}

In this section, we will use the information on the fusion rules of the extension and the $\mathbb{Z}/2\mathbb{Z}$-action established in the previous section to establish the fusion rules for the equivariantization. At a critical juncture we will need to use a braid-to-rotation trick presented in \cite{MR3859960}, which allows us to find certain multiplicities of the fusion rules in terms of multiplicities of eigenvalues of generalized rotation operators. This information can be read off from the modular data of $\mathcal{C}$. First we briefly recall some generalities on equivariantization from \cite{MR3059899}.

Suppose we have a categorical action of $G$ on a tensor category $\mathcal{C}$. For each $g\in G$, we will denote the corresponding monoidal functor $F_{g}$, with tensorator isomorphisms $\mu^{g}_{X,Y}: F_{g}(X)\otimes F_{g}(Y)\rightarrow F_{g}(X\otimes Y)$. The categorical action gives us natural isomorphisms $\rho_{g,h}: F_{g}\circ F_{h}\rightarrow F_{gh}$. 
\begin{defn}\cite[Definition 3.1]{Kirillov_ong-equivariant}
A $G$-equivariant object is a pair $(X, u)$, where $X\in \mathcal{C}$ and $u=\{u_{g}\}_{g\in G}$, where $u_{g}: F_{g}(X)\rightarrow X$ is a natural isomorphism satisfying $u_{gh}\cdot\rho_{g,h}=u_{g}\cdot F_{g}(u_h)$. The category $\mathcal{C}^{G}$ of $G$-equivariant objects and $G$-equivariant morphisms is a tensor category, with $(X, u)\otimes (Y,w) :=(X\otimes Y, u\otimes w) $, where $(u\otimes w)_{g}:= (u_{g}\otimes w_{g})\circ (\mu^{g}_{X,Y})^{-1}$. This category is called the $G$-equivariantization of $\cC$.
\end{defn}

\begin{lem}\cite[Corollary 2.14]{MR3059899}
The simple objects in $\mathcal{C}^{G}$ are given by pairs $(\Gamma, \pi)$, where $\Gamma$ is an orbit of the induced action of $G$ on $Irr(\mathcal{C})$ and $\pi$ is an irreducible projective representation of the isotropy group of a simple object in the orbit (note all isotropy groups are isomorphic). The actual object underlying $(\Gamma, \pi)$ is, as an object, isomorphic to $S_{\Gamma, \pi}:=\pi\otimes \left(\bigoplus_{X\in \Gamma} X\right)$, that is, the direct sum of the simple objects in the orbit, with multiplicity the dimension of the projective representation.
\end{lem}

If $\mathcal{C}$ is a fusion category, its equivariantization $\mathcal{C}^{G}$ is again a fusion category. 
Let $(X,u)\in \mathcal{C}^{G}$ and $Y\in Irr(\mathcal{C})$. The subgroup $G_{Y} = \{g\in G : F_g(Y)\cong Y\}$ is called the isotropy/intertia subgroup of $[Y]$. Then for each $g\in G_{Y}$, we can pick some $c_{g}: F_{g}(Y)\rightarrow Y$. We can define a function $\alpha_Y: G_{Y}\times G_{Y}\rightarrow \mathbb{C}$ by 
$$\alpha_{Y}(g,h) id_{Y}:= c_{gh}\cdot \rho_{g,h}\cdot F_{g}(c^{-1}_{h})\cdot c^{-1}_{g}$$ 

Then $\alpha_{Y}$ is a 2-cocycle on the group $G_{Y}$, and another choice of isomorphisms $c_{g}$ produces a cohomologically equivalent 2-cocycle (see \cite{MR3059899} Section 2.3). Furthermore we define a map $\pi$ on the space $\mathcal{C}(Y, X)$ by 

$$\pi(g)(f):= u_{g}\cdot F_{g}(f)\cdot c^{-1}_{g},$$

 and it is straightforward to verify that $\pi$ is an $\alpha$-projective representation of $G$, that is, this map satisfies 

$$\pi(g)\pi(h)=\alpha(g,h) \pi(gh).$$

To determine the fusion rules in the equivariantization we will use the fact that for any simple object $(\Gamma, \delta)\in \cC^{G}$, the dimension of the space $\mathcal{C}^{G}((\Gamma, \delta), (X, u))$ is the same as the multiplicity of $\delta$ in the projective representation of $G_{Y}$ on the space $\mathcal{C}(Y, X)$ as described above for any simple object $Y\in \Gamma$. Thus to determine fusion rules it is necessary to first identify the cocycle class of $\mathcal{C}(Y, X)$, and then decompose into irreducible representations to determine the multiplicity with which $\delta$ occurs.

Since $H^{2}(\mathbb{Z}/2\mathbb{Z}, \mathbb{C}^{\times})$ is trivial, then all the projective representations that occur can be normalized (via a suitable choice of the $c_{g}$ isomorphisms) to honest representations of $\mathbb{Z}/2\mathbb{Z}$.

Now let $\mathcal{D}:=(\mathcal{C}\boxtimes \mathcal{C})\oplus \widehat{\mathcal{C}}$ be an extension as discussed above. There are three classes of simple objects in $\mathcal{D}^{\mathbb{Z}/2\mathbb{Z}}$ coming from $\mathcal{C}\boxtimes\mathcal{C}$ and two classes of simple objects coming from $\widehat{\mathcal{C}}$. For each $X\in Irr(\mathcal{C})$, we pick a square root of the twist $\theta_{X} = \raisebox{-.5\height}{ \includegraphics[scale = .3]{X_twist}}$ once and for all, and denote it $\theta^{\frac{1}{2}}_{X}$. The other square root is $-\theta^{\frac{1}{2}}_{X}$. The different classes of simple objects of $\mathcal{D}^{\mathbb{Z}/2\mathbb{Z}}$ are described below:

\begin{enumerate}
\item
For $X,Y\in Irr(\mathcal{C})$ and $X\ne Y$, the object $[X,Y]:= (X\boxtimes Y\oplus Y\boxtimes X, u)$, where $u_{1}:=id_{X\boxtimes Y\oplus Y\boxtimes X}$ and $u_{g}=\tau_{\oplus}$, where $\tau_{\oplus}: X\boxtimes Y\oplus Y\boxtimes X\rightarrow Y\boxtimes X\oplus X\boxtimes Y$ is the canonical isomorphism swapping the additive factors.
\item
For $X\in Irr(\mathcal{C})$, $[X,X]_{\pm}:=(X\boxtimes X,u^{\pm})$, where $u^{\pm}_{1}:=id_{X\boxtimes X}$ , $u^{\pm}_{g}:= \pm id_{X\boxtimes X}$.
\item
For $\widehat{X}\in Irr(\widehat{\mathcal{C}})$, $\widehat{X}_{\pm}:=(\widehat{X}, u^{\pm})$, where $u^{\pm}_{1}:=id_{X\boxtimes X}$ and $u^{\pm}_{g}:= \pm \theta^{\frac{1}{2}}_{X} id_{X\boxtimes X}$.
\end{enumerate}

The first class has trivial isotropy group, while the second two has full isotropy groups.

Using commutativity and Frobenius reciprocity combined with the fact that $(X,Y)^{*}=(X^{*},Y^{*})$, $(X,X)^{*}_{\pm}=(X^{*}, X^{*})_{\pm}$, and $(\widehat{X}_{\pm})^{*}=\widehat{X}^{*}_{\pm}$, it turns out that to determine the full set of fusion rules in $\mathcal{D}^{\mathbb{Z}/2\mathbb{Z}}$, we only need to know 3 general cases, which we present in the following theorem. In the following, we adopt the convention that if $\epsilon\in \{\pm\}$, then $\epsilon$ could refer to a subscript of an object (as in $\widehat{X}_{\pm}$) or the numerical value $\pm1$ depending on context.

Below, we let $S_{XY}=\frac{1}{\sqrt{D}} Tr_{X\otimes Y^{*}} (\sigma_{Y^{*}, X} \circ \sigma_{X, Y^{*}})$, where $Tr$ denotes the (unnormalized) spherical trace, and $\sqrt{D}$ is the positive square root of the global dimenion $D:=\sum_{X\in \text{Irr}(\cC)} d^{2}_{X}$.

\begin{thm}\label{fusionequiv} 

\begin{enumerate}
\item For any $Z\in \mathcal{D}^{\mathbb{Z}/2\mathbb{Z}}$,
$$\dim\left(\mathcal{D}^{\mathbb{Z}/2\mathbb{Z}}\left( [X,Y],\ Z\right)\right)=\dim\left(\mathcal{D}\left(X\boxtimes Y, {\bf{G}} (Z)\right)\right),$$ where ${\bf{G}}:\mathcal{D}^{\mathbb{Z}/2\mathbb{Z}}\rightarrow \mathcal{D}$ is the forgetful functor.
\item
For $\epsilon_{X},\epsilon_{Y}, \epsilon_{Z}\in \{\pm \}$,

$$\dim\left(\mathcal{D}^{\mathbb{Z}/2\mathbb{Z}}( [X,X]_{\epsilon_{X}},\ [Y,Y]_{\epsilon_Y}\otimes [Z,Z]_{\epsilon_Z}) \right)=
\frac{1}{2}N^{X}_{YZ}(N^{X}_{YZ}+\epsilon_{X}\epsilon_{Y} \epsilon_{Z}).$$

\item
For $\epsilon_{X},\epsilon_{Y}, \epsilon_{Z}\in \{\pm \}$,

$$\dim\left(\mathcal{D}^{\mathbb{Z}/2\mathbb{Z}}( \widehat{X}_{\epsilon_{X}},\ [Y,Y]_{\epsilon_Y}\otimes \widehat{Z}_{\epsilon_Z}) \right)=$$
$$\frac{1}{2}\left[ \frac{\theta^{\frac{1}{2}}_{Z}}{\theta^{\frac{1}{2}}_{X}} \epsilon_{X}\epsilon_{Y} \epsilon_{Z}\left(\displaystyle\sum_{P,Q\in Irr(\mathcal{C})} S_{Z^{*}P}S_{X^{*}Q}\left(\frac{\theta_{P}}{\theta_{Q}}\right)^{2} N^{Y}_{PQ}\right) +N^{X}_{Y^{2}Z}\right].$$

\end{enumerate}

\end{thm}

\begin{proof}

The first item follows immediately from the fact that the isotropy group associated to the orbit $[X,Y]$ is trivial. 

For the second item recall from Proposition~\ref{prop:action} that $\mathbb{Z}/2\mathbb{Z}$ acts strictly (i.e. $F_{g}\circ F_{g}=id=F_{1}$ and the tensorator $\rho_{g,g}=id$) on the component $\mathcal{C}\boxtimes \mathcal{C}\subseteq \mathcal{D}$. Using the identity on $X\boxtimes X$ for $c_{g}$, the projective representation of $\mathbb{Z}/2\mathbb{Z}$ on $\mathcal{D}(X\boxtimes X, (Y\boxtimes Y)\otimes Z\boxtimes Z)\cong \mathcal{C}(X,Y\otimes Z)\otimes_{\mathbb{C}}\mathcal{C}(X,Y\otimes Z)$ is naturally a honest representation. If $\epsilon_{Y} \epsilon_{Z}=1$, we obtain the ordinary swap representation on the tensor product vector space $\mathcal{C}(X,Y\otimes Z)\otimes_{\mathbb{C}}\mathcal{C}(X,Y\otimes Z)$ while if $\epsilon_{Y} \epsilon_{Z}=-1$, we obtain this representation tensored with the sign representation. Thus using standard dimension formulas for $Sym$ and $Alt$, a case by case analysis yields the desired result.

For the third item we conveniently choose $c_{g}=\theta^{\frac{1}{2}}$. Then, using the tensorator for the $\mathbb{Z}/2\mathbb{Z}$-action from Proposition \ref{prop:action}, the associated (\emph{a-priori} projective) representation on the space $\mathcal{D}(\widehat{X}, (Y\boxtimes Y)\otimes \widehat{Z})=\mathcal{C}(X, YZY)$ is given by 

$$\pi(g)(f)=\epsilon_{Y}\epsilon_{Z}\theta_{Y}\frac{\theta^{\frac{1}{2}}_{Z}}{\theta^{\frac{1}{2}}_{X}} \raisebox{-.5\height}{ \includegraphics[scale = .4]{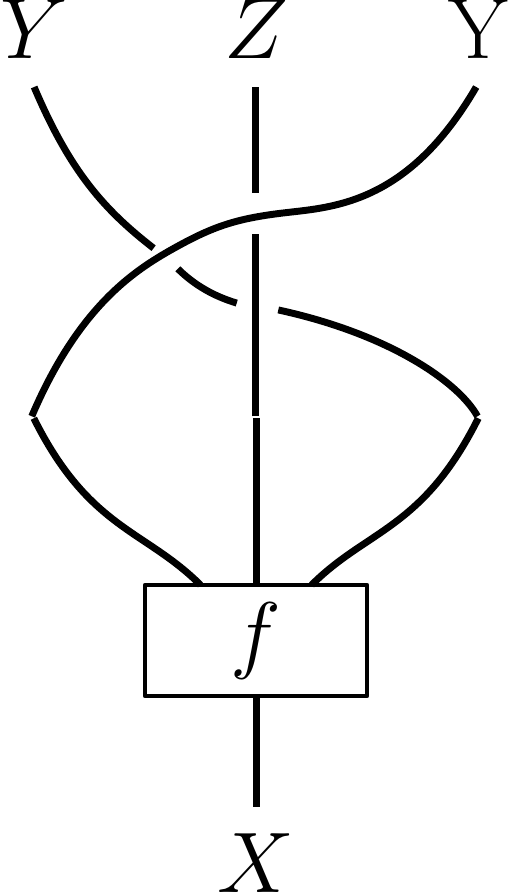}}.$$


We see that our convenient choice of $c_{g}$ has made this representation honest. Now we have an isomorphism $\Psi: \mathcal{C}(X, Y\otimes Z\otimes Y)\rightarrow \mathcal{C}(Z^{*}\otimes X, Y\otimes Y)$, represented via graphical calculus by
 \[ \raisebox{-.5\height}{ \includegraphics[scale = .4]{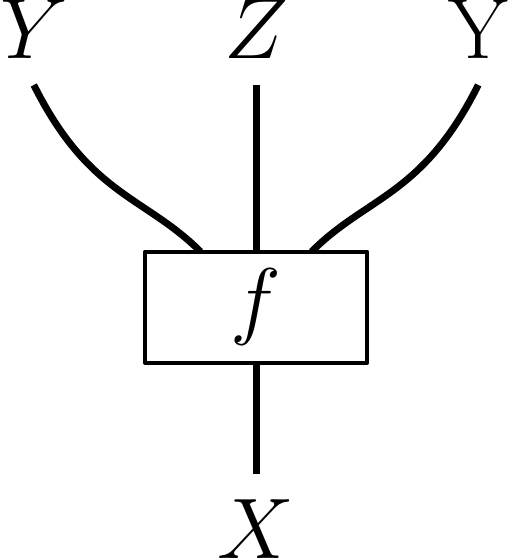}} \overset{\Psi}{\mapsto} \raisebox{-.5\height}{ \includegraphics[scale = .4]{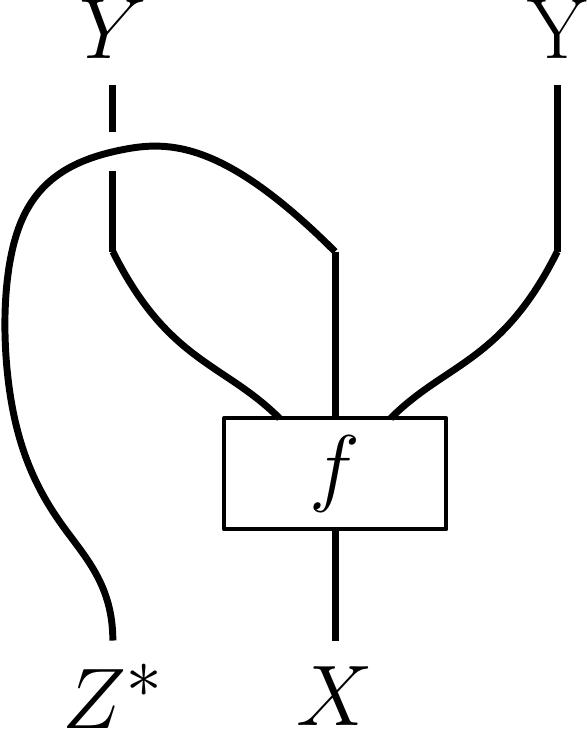}}, \] 
\[ \raisebox{-.5\height}{ \includegraphics[scale = .4]{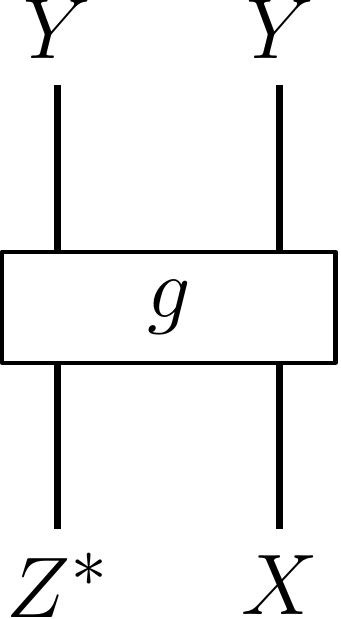}} \overset{\Psi^{-1}}{\mapsto} \raisebox{-.5\height}{ \includegraphics[scale = .4]{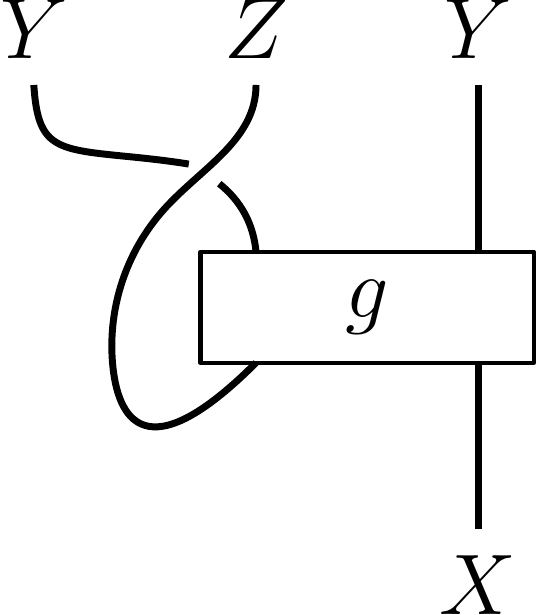}}. \] 

Equipping $Z^{*}\otimes W$ with the half-braiding by using the braiding on $Z^{*}$ (represented graphically by an over-braiding) and the inverse braiding on $W$ (represented graphically by an under-braiding), we can view $Z^{*}\otimes W$ as the simple object $Z^{*}\boxtimes X\in \mathcal{C}\boxtimes \mathcal{C}^{rev}\cong Z(\mathcal{C})$ \cite{Muger-Oxford}. We consider the generalized rotation operator $R:\mathcal{C}(Z^{*}\otimes X, Y\otimes Y)\rightarrow \mathcal{C}(Z^{*}\otimes X, Y\otimes Y)$ that is represented graphically by: 
\[ \raisebox{-.5\height}{ \includegraphics[scale = .4]{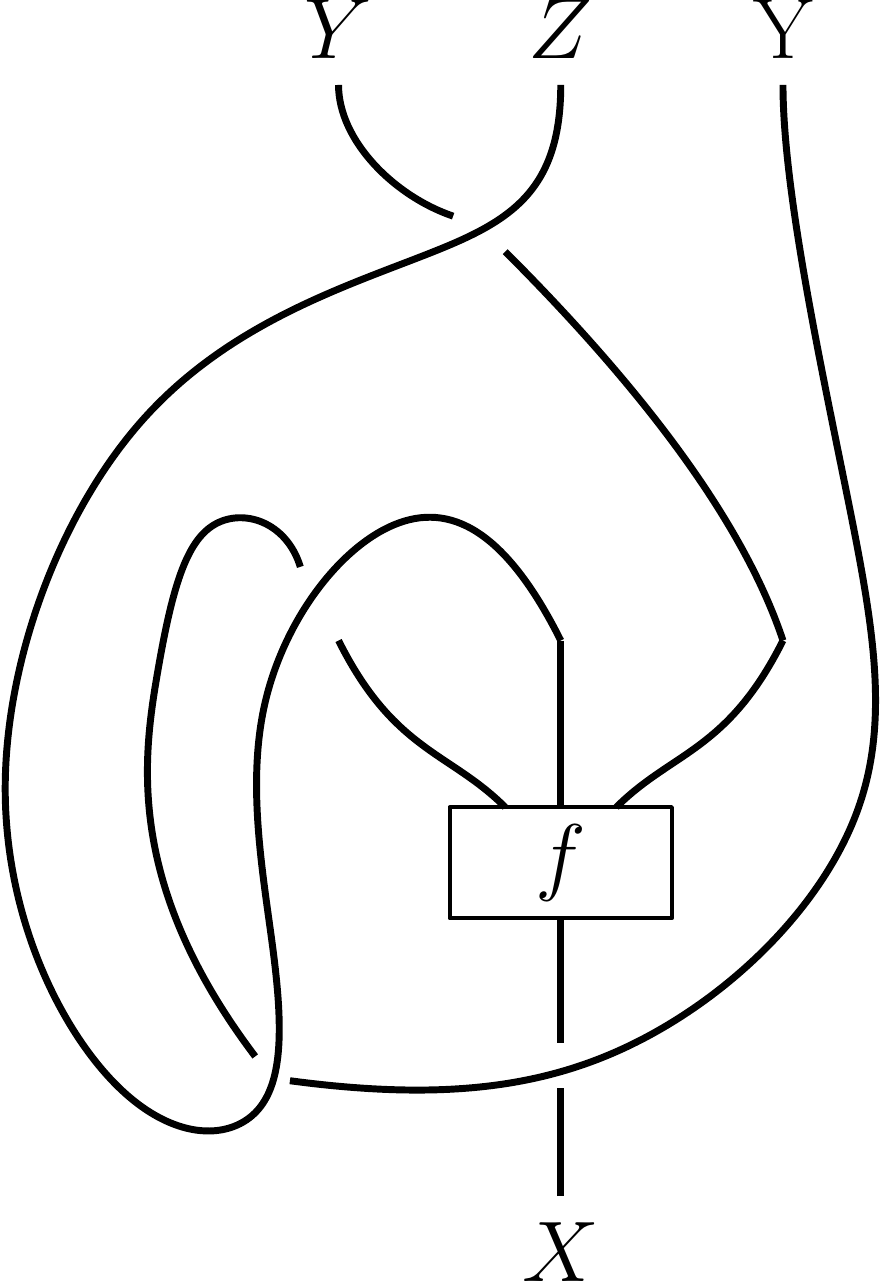}} = \raisebox{-.5\height}{ \includegraphics[scale = .4]{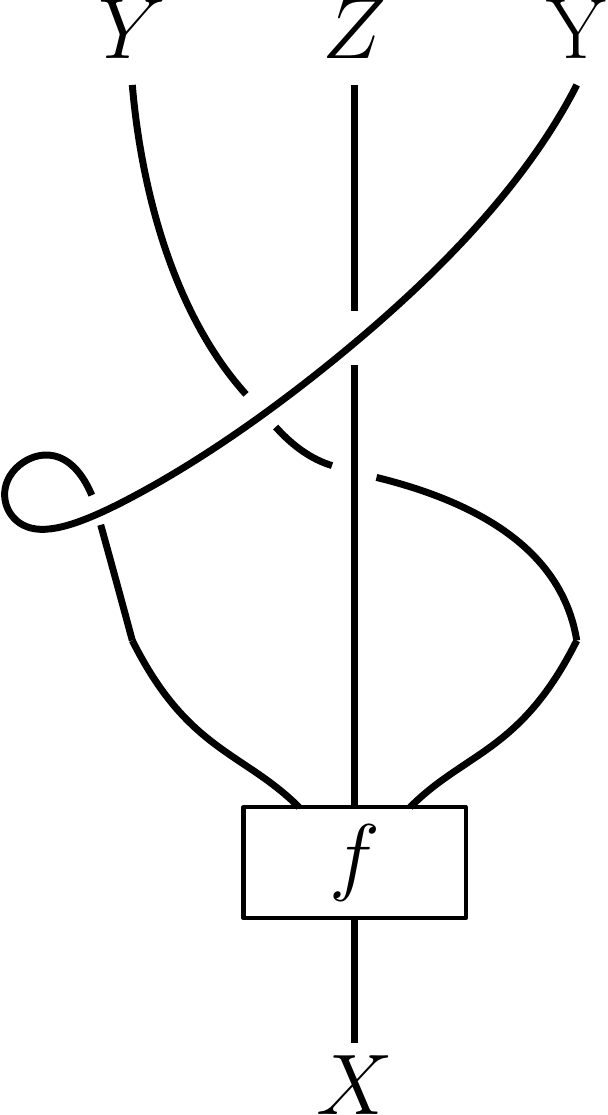}}, \] 
first introduced by Ng and Schauenburg \cite{MR2725181}. The above diagram shows that
$$\pi(g)=\epsilon_{Y}\epsilon_{Z}\theta^{\frac{1}{2}}_{Z} \theta^{-\frac{1}{2}}_{X} \Psi^{-1} \circ R \circ \Psi.$$

Thus to determine the multiplicities of the eigenvalues of $-1$ and $1$ respectively, it suffices to determine the multiplicities of the eigenvalues $\pm\epsilon_{Y}\epsilon_{Z}\theta^{-\frac{1}{2}}_{Z} \theta^{\frac{1}{2}}_{X}$. By Proposition 4.2 of \cite{MR3859960}, the result follows. For the reader unfamiliar with this result, it follows from application of Ng and Schauenburg's generalized Frobenius-Schur indicator formulas \cite{MR2725181} (see \cite{AnRanThesis} for a detailed exposition), combined with a finite Fourier transform.
\end{proof}

\subsection{Comparison with formulas from conformal field theory}

In \cite[Equations 4.36-4.38]{MR1606821} and \cite[Equations 40-44]{MR2116735}, the authors give the fusion rules for the modules of a $\mathbb{Z}/2\mathbb{Z}$ permutation orbifold rational conformal field theory. By \cite{MR2183964}, the category of modules of the orbifold is the $\mathbb{Z}/2\mathbb{Z}$ permutation gauging. As a consistency check, we show that our formulas and theirs are actually the same. To describe the formulas in these two references, we first must translate notations and conventions.

In \cite{MR1606821}, there are 3 types of simple objects: $(ij), (i \psi), \widehat{(i\psi)}$, where $i,j\in \text{Irr}(\cC)$, and $\psi\in \{0,1\}$. In our notation $(ij)$ would correspond to objects of the form $[X,Y]$, $(i \psi)$ corresponds to objects $[Y,Y]_{\epsilon_{Y}}$, and $\widehat{(i,\psi)}$ corresponds to the objects $\widehat{X}_{\epsilon_{X}}$, where in all the above cases $\epsilon=e^{\pi i \psi}$. With this dictionary, by inspection our formulas for the fusion rules agree with the formulas from equations 4.36 and 4.37 in \cite{MR1606821}. However, its not clear a-priori whether our item (3) in Theorem \ref{fusionequiv} agrees with their equation 4.38. Below, we will show this is indeed the case.

Let $\zeta=(\frac{p_{+}}{p_{-}})^{\frac{1}{6}}$, where $\displaystyle p_{\pm}=\sum_{X\in \text{Irr}(\cC)}d^{2}_{X}\theta^{\pm}_{X}$ are the Gauss sums. In the conformal field theory context, $\zeta$ is related to the central charge $c$ by $\zeta=e^{\frac{2\pi i c}{24}}$ (see \cite[Remark 3.1.20]{MR1797619}).

Then with $T:=\delta_{X,Y} \theta_{X}$ and defining $\hat{T}:=\frac{T}{\zeta}$, we get an honest representations of the modular group. In the references \cite{MR1606821}, \cite{MR2116735}, what they call $T$ is what we call $\hat{T}$ here, but $S$ has the same meaning here as there (recall we define $S$ immediately preceding the statement of Theorem \ref{fusionequiv}). If we let $C_{XY}=\delta_{X,Y^{*}}$ denote the charge conjugation matrix, we have the relations

$$(S\hat{T})^{3}=S^{2}=C,\ C^{2}=1,\ \hat{T}C=C \hat{T}.$$

 $S$ is a symmetric unitary. This (and the above relations) imply $(S^{-1})_{XY}=S_{X^{*}Y}=S_{XY^{*}}$. We also have the \textit{Verlinde formula} 

$$N^{Z}_{XY}=\sum_{R\in \text{Irr}(\cC)} \frac{S_{XR} S_{YR} S_{Z^{*}R}}{S_{1R}}$$

For all of the above, see \cite[Chapter 3]{MR1797619}. Now to describe the formulas for fusion rules from \cite{MR1606821}, define the matrix $P:=\hat{T}^{\frac{1}{2}} S \hat{T}^{2} S \hat{T}^{\frac{1}{2}}$. Then translating their equation 4.38 into our language, we have

$$\dim\left(\mathcal{D}^{\mathbb{Z}/2\mathbb{Z}}( \widehat{X}_{\epsilon_{X}},\ [Y,Y]_{\epsilon_Y}\otimes \widehat{Z}_{\epsilon_Z}) \right)=$$
$$\frac{1}{2}\left[ \sum_{R\in \text{Irr}(\cC)} \frac{S^{2}_{YR}S_{ZR}S_{X^{*}R}}{S^{2}_{1R}}+\epsilon_{X}\epsilon_{Y}\epsilon_{Z} \sum_{R\in \text{Irr}(\cC)} \frac{S_{YR}P_{ZR}P_{X^{*}R}}{S_{1R}}\right]$$

We claim the two terms on the right hand side of this equation match our two terms. Consider the first term. Recall we have the equation (which follows from the Verlinde formula)

$$\frac{S_{XR}S_{YR}}{S_{1R}}=\sum_{Z\in \text{Irr}(\cC)} N^{Z}_{XY}S_{ZR}$$

Applying this (along with the usual Verlinde formula) we get 

$$\sum_{R\in \text{Irr}(\cC)} \frac{S^{2}_{YR}S_{ZR}S_{X^{*}R}}{S^{2}_{1R}}=\sum_{R\in \text{Irr}(\cC)} \frac{N^{W}_{YY}S_{WR}S_{ZR}S_{X^{*}R}}{S_{1R}}=\sum_{W\in \text{Irr}(\cC)}N^{W}_{YY}N^{X}_{WZ}=N^{X}_{Y^{2}Z},$$

which is exactly the second term in our formula. Now, using the Verlinde formula and renormalizing, we can write the first term from item $(3)$ in Theorem \ref{fusionequiv} as

$$\frac{\theta^{\frac{1}{2}}_{Z}}{\theta^{\frac{1}{2}}_{X}} \epsilon_{X}\epsilon_{Y} \epsilon_{Z}\left(\displaystyle\sum_{P,Q\in Irr(\mathcal{C})} S_{Z^{*}P}S_{X^{*}Q}\left(\frac{\theta_{P}}{\theta_{Q}}\right)^{2} N^{Y}_{PQ}\right)$$
$$=\epsilon_{X}\epsilon_{Y}\epsilon_{Z} \sum_{R} \frac{(\hat{T}^{\frac{1}{2}}S^{-1}\hat{T}^{2}S)_{ZR}\ (\hat{T}^{-\frac{1}{2}}S^{-1} \hat{T}^{-2} S)_{XR}\ S^{-1}_{YR}}{S_{1R}}.$$

We note that because the $\cC^{\boxtimes 2}$ and its gauging are braided fusion categories, we can replace every object with its dual in the formula for fusion multiplicity since $N^{Z}_{XY}=N^{Z^{*}}_{X^{*}Y^{*}}$. This change makes the above formula into

$$=\epsilon_{X}\epsilon_{Y}\epsilon_{Z} \sum_{R} \frac{(\hat{T}^{\frac{1}{2}}S \hat{T}^{2}S)_{ZR}\ (\hat{T}^{-\frac{1}{2}}S \hat{T}^{-2} S)_{XR}\ S_{YR}}{S_{1R}}.$$

Note that by definition, we have $\hat{T}^{\frac{1}{2}}S \hat{T}^{2}S=P\hat{T}^{-\frac{1}{2}}$.  From the modular group relations, we have $\hat{T}S\hat{T}=S\hat{T}^{-1}S$, hence 

$$P\hat{T}^{\frac{1}{2}}=\hat{T}^{\frac{1}{2}} S \hat{T}^{2} S \hat{T}=\hat{T}^{-\frac{1}{2}} (\hat{T}S\hat{T}) (\hat{T} S \hat{T})=\hat{T}^{-\frac{1}{2}}(S\hat{T}^{-1} S)(S\hat{T}^{-1} S)=\hat{T}^{-\frac{1}{2}}SC\hat{T}^{-2}S=\hat{T}^{-\frac{1}{2}}S^{-1}\hat{T}^{-2}S.$$

Thus substituting into the above expression, the above expression becomes

$$=\epsilon_{X}\epsilon_{Y}\epsilon_{Z} \sum_{R} \frac{(P\hat{T}^{-\frac{1}{2}})_{ZR}\ (P\hat{T}^{\frac{1}{2}})_{X^{*}R}\ S_{YR}}{S_{1R}}$$
$$=\epsilon_{X}\epsilon_{Y}\epsilon_{Z} \sum_{R} \left(\frac{\theta_{R}}{\zeta}\right)^{\frac{1}{2}}\left(\frac{\theta_{R}}{\zeta}\right)^{-\frac{1}{2}}\frac{P_{ZR}\ P_{X^{*}R}\ S_{YR}}{S_{1R}}.$$
$$\epsilon_{X}\epsilon_{Y}\epsilon_{Z} \sum_{R} \frac{P_{ZR}\ P_{X^{*}R}\ S_{YR}}{S_{1R}},$$

as desired.

\section{Examples of permutation gauging}
To end this paper we explicitly compute the fusion rules for the $\mathbb{Z}/2\mathbb{Z}$ permutation gauging for several examples of modular tensor categories. As the rank of the resulting modular tensor categories grows quadratically, we restrict our attention to low rank categories. The full fusion rules of these categories are too big to include in this paper, hence we only include the graph representing fusion with the distinguished object $\widehat{\mathbf{1}}_+$. Our examples will consist of modular tensor categories constructed from level $k$ integrable representations of an affine Lie algebra $\widehat{\mathfrak{g}}$, which we denote $\cC(\mathfrak{g}, k)$. We direct the reader towards \cite{quan-primer} for details on these modular tensor categories, including explicit formulas for the modular data which we will need for our example computations. Attached to the arXiv submission of this paper are plain text files containing the full fusion ring for the third example, the $\mathbb{Z}/2\mathbb{Z}$ permutation gauging of the category $\cC(\mathfrak{g}_2, 3)$. 

\subsection*{The core of $\cC(\mathfrak{sl}_2, 8)$} Consider the modular tensor category $\cC(\mathfrak{sl}_2, 8)$. This category  contains the commutative algebra object $\mathbf{1} \oplus (7\Lambda_1)$. Let $\cC$ be the modular tensor category of dyslectic $\mathbf{1} \oplus (7\Lambda_1)$-modules in $\cC(\mathfrak{sl}_2, 8)$. This category $\cC$ is known as the core of $\cC(\mathfrak{sl}_2, 8)$. For details on the category of dyslectic modules, we point the reader to \cite[Section 3.12]{MR3039775}. We remark for those from an operator algebraic background, that $\cC$ can also be realised as the even part of the  $D_6$ subfactor standard invariant \cite{MR2559686}.

The fusion graph for the object $\widehat{\mathbf{1}}_+$ in the $\mathbb{Z}/2\mathbb{Z}$ permutation gauging of $\cC$ is as follows:
\begin{figure}[H]
\raisebox{-.5\height}{ \includegraphics[scale = 1]{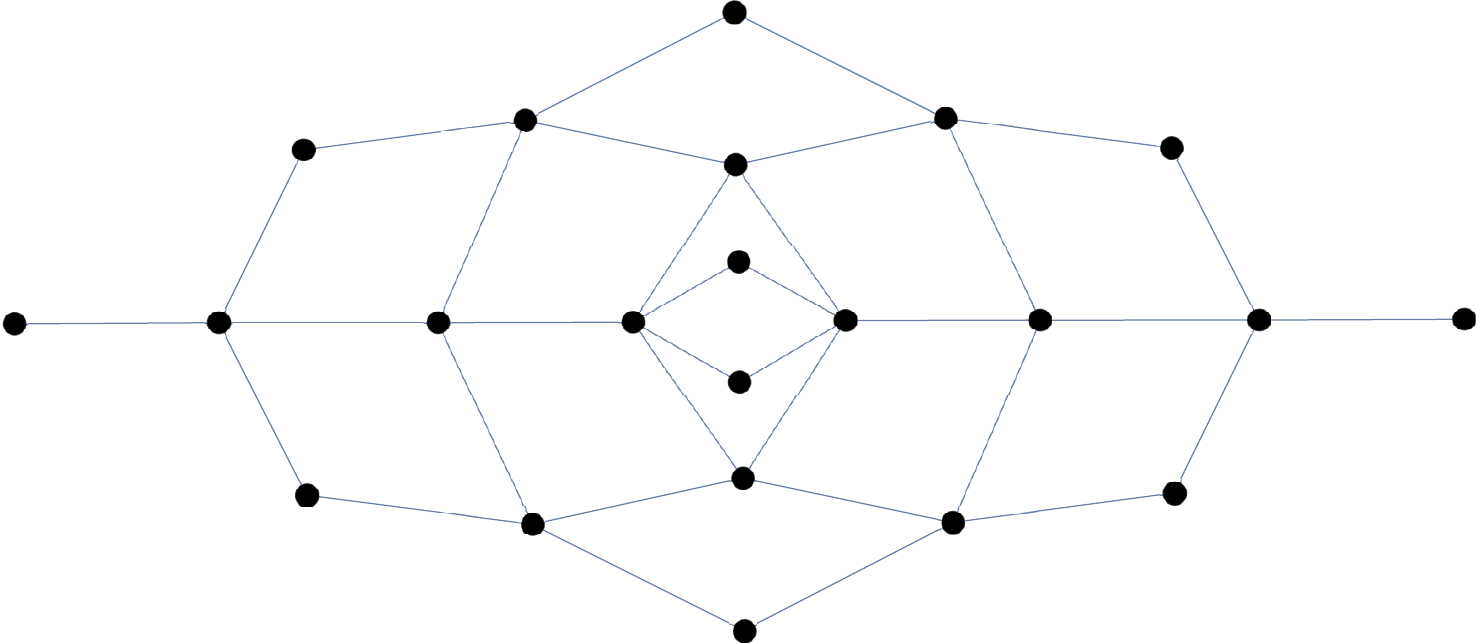}}
\caption{Fusion graph of the object $\widehat{\mathbf{1}}_+$ in the $\mathbb{Z}/2\mathbb{Z}$ permutation gauging of the core of $\cC(\mathfrak{sl}_2, 8)$.}
\end{figure}

This resulting category has the same fusion rules as the category of dyslectic $\mathbf{1} \oplus (4\Lambda_1)$-modules in $\cC(\mathfrak{so}_8,4)$. We observe similar behavior when we study the $\mathbb{Z}/2\mathbb{Z}$ permutation gauging of the category of dyslectic $\mathbf{1} \oplus (11\Lambda_1)$-modules in $\cC(\mathfrak{sl}_2, 12)$, which leads us to make the following conjecture. Recall $\cC^\text{rev}$ means the category monoidally equivalent to $\cC$, with the reverse braiding \cite[Definition 8.1.4]{MR3242743}.

\begin{conj}
Let $\cC$ be the modular tensor category of dyslectic $\mathbf{1} \oplus ((4N-1)\Lambda_1)$-modules in  $\cC(\mathfrak{sl}_2, 4N)$. Then the $\mathbb{Z}/2\mathbb{Z}$ permutation gauging of $\cC$ is braided equivalent to the category of dyslectic $\mathbf{1} \oplus (4\Lambda_1)$-modules in $\cC(\mathfrak{so}_{2N+4},4)^\text{rev}$.
\end{conj}

It is interesting to note that while the initial modular category in this example is not prime (it decomposes as $\text{Fib}\boxtimes\text{Fib}$), the category constructed by the $\mathbb{Z}/2\mathbb{Z}$ permutation gauging is prime.

\subsection*{The adjoint subcategory of $\cC(\mathfrak{sl}_2, 5)$} Consider the modular tensor category $\cC(\mathfrak{sl}_2, 5)$. This category contains a rank 3 modular tensor subcategory, $\otimes$-generated by the simple object $(2\Lambda_1)$. Let $\cC$ denote this rank 3 modular tensor subcategory. This category $\cC$ is known as the adjoint subcategory of $\cC(\mathfrak{sl}_2, 5)$. We remark that $\cC$ can also be realised as the even part of the $A_{6}$ subfactor standard invariant. 

The fusion graph for the object $\widehat{\mathbf{1}}_+$ in the permutation gauging of $\cC$ is as follows:

\begin{figure}[H]
\raisebox{-.5\height}{ \includegraphics[scale = 1]{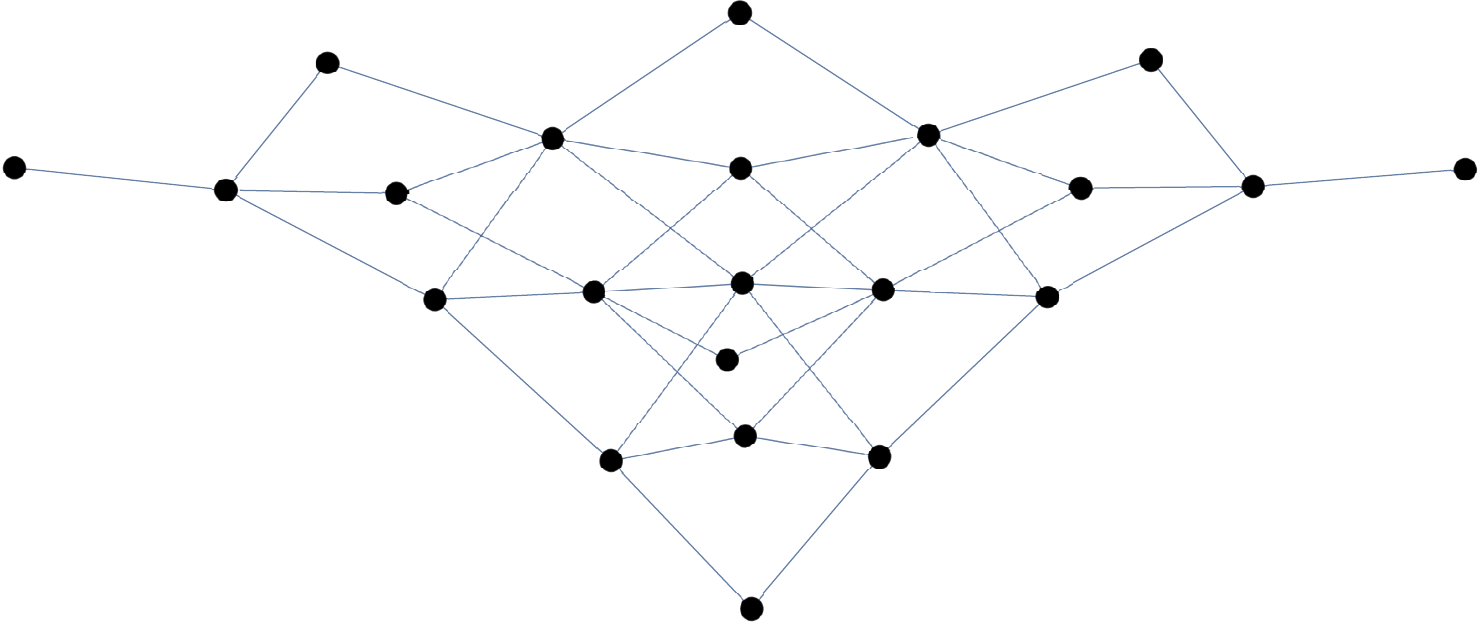}}
\caption{Fusion graph of the object $\widehat{\mathbf{1}}_+$ in the $\mathbb{Z}/2\mathbb{Z}$ permutation gauging of the adjoint subcategory of $\cC(\mathfrak{sl}_2, 5)$.}
\end{figure}

This modular tensor category has the same fusion rules as $\cC(\mathfrak{so}_5, 4)$. We notice similar behavior when we take the $\mathbb{Z}/2\mathbb{Z}$ permutation gauging of the modular tensor subcategory of $\cC(\mathfrak{sl}_2,2N+1)$ $\otimes$-generated by the object $(2\Lambda_1)$, for small $N$. Hence we are lead to make the following conjecture.

\begin{conj}
Let $\cC$ be the modular tensor subcategory of $\cC(\mathfrak{sl}_2,2N+1)$ $\otimes$-generated by the object $(2\Lambda_1)$, then the $\mathbb{Z}/2\mathbb{Z}$ permutation gauging of $\cC$ is braided equivalent to $\cC(\mathfrak{so}_{2N+1},4)^\text{rev}$.
\end{conj}

\subsection*{The category $\cC(\mathfrak{g}_2,3)$}
Consider the modular tensor category $\cC(\mathfrak{g}_2,3)$. The fusion graph for the object $\widehat{\mathbf{1}}_+$ in the permutation gauging of $\cC(\mathfrak{g}_2,3)$ is as follows:

\begin{figure}[H]
\raisebox{-.5\height}{ \includegraphics[scale = 1]{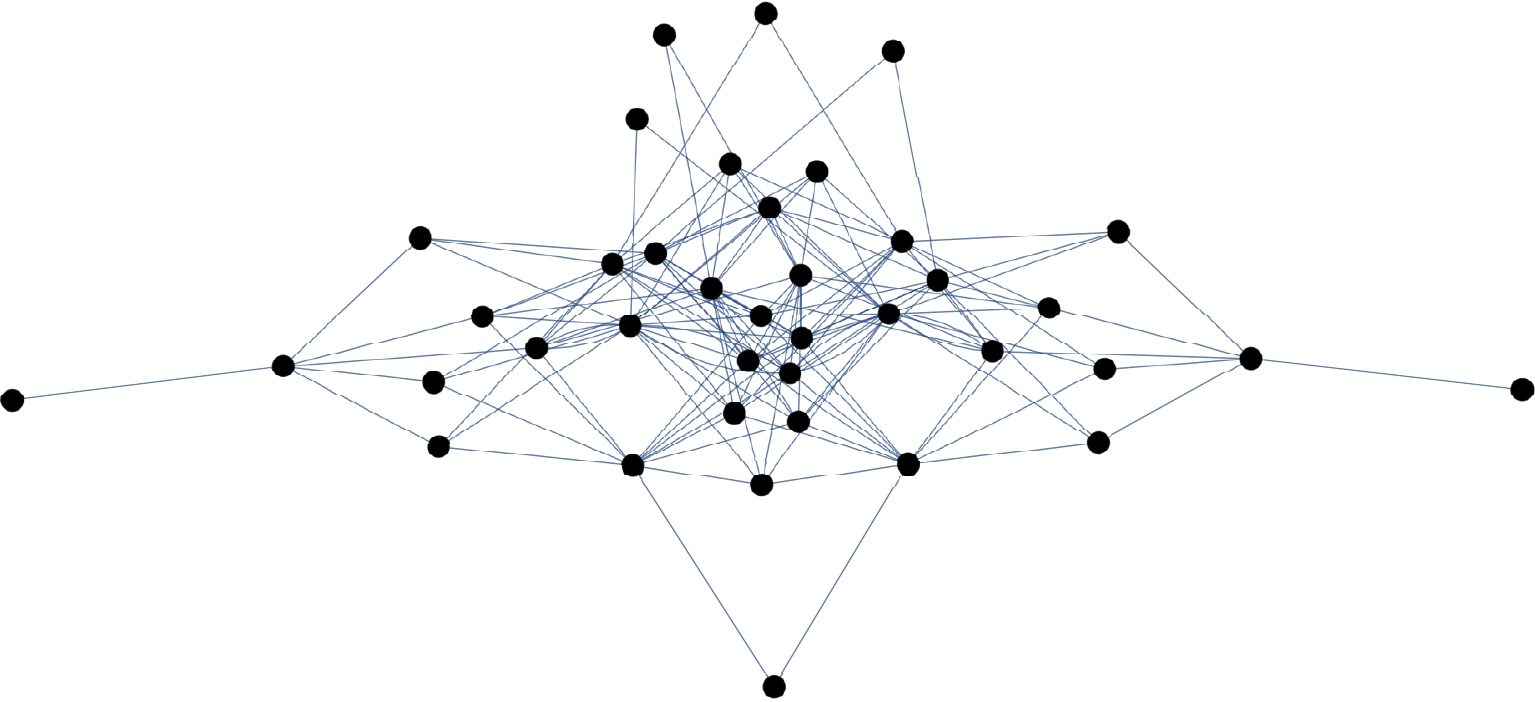}}
\caption{Fusion graph of the object $\widehat{\mathbf{1}}_+$ in the $\mathbb{Z}/2\mathbb{Z}$ permutation gauging of $\cC(\mathfrak{g}_2, 3)$.}
\end{figure}

The dimensions of the 39 simple objects in this category are:

\begin{align*}
\{ &1,1,\sqrt{21}+3,\sqrt{21}+3,\sqrt{21}+3,\sqrt{21}+5,\frac{3}{2} \left(\sqrt{21}+5\right),\frac{3}{2} \left(\sqrt{21}+5\right),\frac{3}{2} \left(\sqrt{21}+5\right) \\
&,\frac{3}{2} \left(\sqrt{21}+5\right), \frac{3}{2} \left(\sqrt{21}+5\right),\frac{3}{2} \left(\sqrt{21}+5\right), 3 \left(\sqrt{21}+5\right),3 \left(\sqrt{21}+5\right),3 \left(\sqrt{21}+5\right),\\
&\frac{7}{2} \left(\sqrt{21}+5\right),\frac{7}{2} \left(\sqrt{21}+5\right),\sqrt{\frac{21}{2} \left(\sqrt{21}+5\right)},\sqrt{\frac{21}{2} \left(\sqrt{21}+5\right)},\sqrt{21}+7,4 \sqrt{21}+18, \\
&4 \sqrt{21}+18,4 \sqrt{21}+18,5 \sqrt{21}+21,5 \sqrt{21}+21,5 \sqrt{21}+21,\frac{1}{2} \left(5 \sqrt{21}+23\right),\frac{1}{2} \left(5 \sqrt{21}+23\right),\\
&3 \sqrt{\frac{7}{2} \left(5 \sqrt{21}+23\right)},3 \sqrt{\frac{7}{2} \left(5 \sqrt{21}+23\right)},3 \sqrt{\frac{7}{2} \left(5 \sqrt{21}+23\right)},3 \sqrt{\frac{7}{2} \left(5 \sqrt{21}+23\right)},\\
&3 \sqrt{\frac{7}{2} \left(5 \sqrt{21}+23\right)},3 \sqrt{\frac{7}{2} \left(5 \sqrt{21}+23\right)},6 \sqrt{21}+28,\sqrt{21 \left(12 \sqrt{21}+55\right)},\sqrt{21 \left(12 \sqrt{21}+55\right)},\\
&7 \sqrt{  \frac{15}{2} \sqrt{21}+\frac{69}{2}},7 \sqrt{  \frac{15}{2} \sqrt{21}+\frac{69}{2}}\} .
\end{align*}

To the best of our knowledge, this modular tensor category is new.

\vspace{2em}

\textbf{Acknowledgements.} This work began during the third author's visit to the Australian National University, and she would like to thank the University and Scott Morrison for their hospitality. The authors would like to thank Marcel Bischoff, C\'esar Galindo, Scott Morrison, Andrew Schopieray, and Zhenghang Wang for many useful discussions.
Julia Plavnik was supported by the AMS-Simons travel grant. Cain Edie-Michell and Corey Jones were supported by Discovery Projects ‘Subfactors and symmetries’ DP140100732 and ‘Low dimensional categories’ DP160103479 from the Australian Research Council.

\bibliographystyle{aip}
\raggedright
\bibliography{bibliography}

\end{document}